\documentclass[12pt]{amsart}
\usepackage{graphicx,amssymb}
\usepackage[usenames]{color}
\usepackage{amsthm,amsfonts,amsmath,amssymb,latexsym,epsfig,mathrsfs,yfonts,marvosym,latexsym,epsfig}
\usepackage[all]{xy}
\AtEndDocument{\vfill\eject\batchmode}
\DeclareMathAlphabet\oldmathcal{OMS}        {cmsy}{b}{n}
\SetMathAlphabet    \oldmathcal{normal}{OMS}{cmsy}{m}{n}
\DeclareMathAlphabet\oldmathbcal{OMS}       {cmsy}{b}{n}
\usepackage[active]{srcltx}
\usepackage{eucal}

\newtheorem{theorem}{Theorem}[section]

\newtheorem{proposition}[theorem]{Proposition}

\newtheorem{def/prop}[theorem]{Definition/Proposition}

\theoremstyle{definition}
\newtheorem{definition}[theorem]{Definition}
\newtheorem{remark}[theorem]{Remark}
\newtheorem*{ack}{Acknowledgements}
\newtheorem{example}{Example}[section]

\DeclareSymbolFont{bbold}{U}{bbold}{m}{n}
\DeclareSymbolFontAlphabet{\mathbbold}{bbold}

\def\BOne{\mathchoice{\scalebox{1.16}{$\displaystyle\mathbbold 1$}}{\scalebox{1.16}{$\textstyle\mathbbold 1$}}{\scalebox{1.16}{$\scriptstyle\mathbbold 1$}}{\scalebox{1.16}{$\scriptscriptstyle\mathbbold 1$}}}
\def\fract#1#2{\raise4pt\hbox{$ #1 \atop #2 $}}

\def\bbc{{\mathbb C}}

\def\bbn{{\mathbb N}}

\def\bbp{{\mathbb P}}
\def\bbq{{\mathbb Q}}
\def\bbr{{\mathbb R}}

\def\bbt{{\mathbb T}}

\def\bbz{{\mathbb Z}}

\def\gra{\alpha}

\def\gre{\epsilon}

\def\grk{\kappa}

\def\gro{\omega}

\def\grr{\rho}

\def\grD{\Delta}

\def\grO{\Omega}

\def\grS{\Sigma}

\def\bfk{{\bf k}}

\def\bfn{{\bf n}}

\def\bfw{{\bf w}}

\def\cala{{\mathcal A}}

\def\cald{{\mathcal D}}

\def\calo{{\mathcal O}}

\def\cals{{\oldmathcal S}}

\def\calw{{\mathcal W}}

\def\calS{{\mathcal S}}

\def\gt{{\mathfrak t}}

\def\gw{{\mathfrak w}}

\def\gz{{\mathfrak z}}

\def\gC{{\mathfrak C}}

\def\gR{{\mathfrak R}}

\def\lra{\longrightarrow}

\def\<{\langle}
\def\>{\rangle}
\def\ra#1{\to}

\def\fract#1#2{\raise4pt\hbox{$ #1 \atop #2 $}}
\def\decdnar#1{\phantom{\hbox{$\scriptstyle{#1}$}}
\left\downarrow\vbox{\vskip15pt\hbox{$\scriptstyle{#1}$}}\right.}

\def\hook{\mathbin{\hbox to 6pt{%
                 \vrule height0.4pt width5pt depth0pt
                 \kern-.4pt
                 \vrule height6pt width0.4pt depth0pt\hss}}}

\begin{document}

\title{Sasakian Geometry on Sphere Bundles II: Constant Scalar Curvature}

\author[Charles Boyer]{Charles P. Boyer}
\address{Charles P. Boyer, Department of Mathematics and Statistics,
University of New Mexico, Albuquerque, NM 87131, USA.}
\email{cboyer@math.unm.edu} 
\author[Christina T{\o}nnesen-Friedman]{Christina W. T{\o}nnesen-Friedman}
\address{Christina W. T{\o}nnesen-Friedman, Department of Mathematics, Union
College, Schenectady, New York 12308, USA }
\email{tonnesec@union.edu}
\thanks{The authors were partially supported by grants from the
Simons Foundation, CPB by (\#519432), and CWT-F by (\#422410)}
\date{\today}

\begin{abstract}
In a previous paper \cite{BoTo20b} the authors employed the fiber join construction of Yamazaki \cite{Yam99} together with the admissible construction of Apostolov, Calderbank, Gauduchon, and T{\o}nnesen-Friedman \cite{ACGT08} to construct new extremal Sasaki metrics on odd dimensional sphere bundles over smooth projective algebraic varieties. In the present paper we continue this study by applying a recent existence theorem \cite{BHLT22} that shows that under certain conditions one can always obtain a constant scalar curvature Sasaki metric in the Sasaki cone. Moreover, we explicitly describe this construction for certain sphere bundles of dimension 5 and 7.

\end{abstract}

\maketitle\
\vspace{-7mm}



\section{Introduction}\label{intro}
A central problem in Riemannian geometry is to determine conditions for a metric to have constant scalar curvature. This is particularly true in K\"ahler geometry as well as in its odd dimensional sister Sasaki geometry. Specifically, we combine our construction of extremal Sasaki metrics on odd dimensional sphere bundles \cite{BoTo20b} using Yamazaki's fiber join \cite{Yam99} with the admissible conditions \cite{ACGT08} as applied in \cite{BHLT22} to obtain constant scalar curvature (CSC) Sasaki metrics. This involves the introduction of a refinement of the admissibility conditions that we call {\it strongly admissible} whose precise definition is given below in \ref{stradmdef}. Previously \cite{BoTo20b}, we gave a stronger condition called {\it super admissible}; however, we show here that the less stringent condition, strongly admissible, is enough. Explicitly, in Section \ref{mainthm} we prove our main theorem which is Theorem \ref{2ndroundexistence} and is restated here for the convenience of the reader.

\begin{theorem}\label{2ndroundexistenceintro}
Let $M_\gw$ be a strongly admissible Yamazaki fiber join whose regular quotient is a ruled manifold of the form $\bbp(E_0\oplus E_\infty)\longrightarrow N$ where $E_0 ,E_\infty$ are projectively flat hermitian holomorphic vector bundles on $N$ of complex dimension $(d_0+1),(d_\infty+1)$ respectively, and $N$ is a local K\"ahler product of non-negative CSC metrics. Then $\gt^+_{sph}$ has a $2$-dimensional subcone of extremal Sasaki metrics (up to isotopy) which contains at least one ray of CSC Sasaki metrics.
\end{theorem}

Then in Section \ref{Exsect} we present detailed descriptions of strongly admissible sphere bundles of dimension 5 and 7 which are obtained by Yamazaki's fiber join construction. We obtain existence results (Propositions \ref{Mkprop}, \ref{2highergenus} and \ref{polyprop}) of CSC and extremal Sasaki metrics even when the conditions of Theorem \ref{2ndroundexistenceintro} are not all met.

\begin{ack}
The authors would like to thank Claude LeBrun for pointing out \cite{Shu}.  We also would like to thank Eveline Legendre  and Hongnian Huang for fruitful conversations.
\end{ack}

\section{Brief Review of K-Contact and Sasaki Geometry}
Recall that an oriented and co-oriented contact manifold $M^{2n+1}$ has a contact metric structure $\cals=(\xi,\eta,\Phi,g)$ where $\eta$ is a contact form with contact bundle $\cald=\ker\eta$, $\xi$ is its Reeb vector field, $J=\Phi |_\cald$ is an almost complex structure on $\cald$, i.e. $(\cald,J)$ is an almost CR contact structure, and $g=d\eta\circ (\BOne \times\Phi) +\eta\otimes\eta$ is a compatible Riemannian metric. $\cals$ is {\it K-contact}  if $\xi$ is a Killing vector field for $g$, and it is {\it Sasakian} if in addition the almost CR structure is integrable. A manifold with a K-contact (Sasakian) structure is called a {\it K-contact (Sasaki) manifold}. Unless otherwise stated we shall assume that our contact manifolds $M^{2n+1}$ are oriented, co-oriented, compact, connected, and without boundary. We refer to \cite{BG05} for the fundamentals of Sasaki geometry. 


\subsection{The Sasaki Cone}
Let $(M,\cals)$ be a K-contact manifold. Within the underlying contact almost CR structure $(\cald,J)$ there is a conical family of K-contact structures known as the Sasaki cone and denoted by $\gt^+(\cald,J)$ or just $\gt^+$ when the underlying almost CR structure is understood. We are also interested in a variation within this family. To describe the Sasaki cone we fix a K-contact structure $\cals_o=(\xi_o,\eta_o,\Phi_o,g_o)$ on $M$ whose underlying CR structure is $(\cald,J)$ and let $\gt$ denote the Lie algebra of a maximal torus $\bbt$ in the automorphism group of $\cals_o$. Since for K-contact (Sasakian) structures, the Reeb vector field $\xi$ is a Killing vector field, we have $\dim\gt^+(\cald,J)\geq 1$.  Moreover, it follows from contact geometry that $\dim\gt^+(\cald,J)\leq n+1$. The {\it (unreduced) Sasaki cone} \cite{BGS06} is defined by
\begin{equation}\label{sascone}
\gt^+(\cald,J)=\{\xi\in\gt~|~\eta_o(\xi)>0~\text{everywhere on $M$}\},
\end{equation}
which is a cone of dimension $k\geq 1$ in $\gt$ under the transverse scaling operation defined by
\begin{equation}\label{transscale}
\cals=(\xi,\eta,\Phi,g)\mapsto \cals_a=(a^{-1}\xi,a\eta,\Phi,g_a),\quad g_a=ag+(a^2-a)\eta\otimes\eta, \quad a\in\bbr^+.
\end{equation}
We remark also that $\cals_a$ is a K-contact or Sasakian structure for all $a\in\bbr^+$. The reduced Sasaki cone $\grk(\cald,J)$ is $\gt^+(\cald,J)/\calw$ where $\calw$ is the Weyl group of the maximal compact subgroup of $\gC\gR(\cald,J)$ which, as described in \cite{BGS06}, is the moduli space of K-contact (Sasakian) structures with underlying CR structure $(\cald,J)$. However, it is more convenient to work with the unreduced Sasaki cone $\gt^+(\cald,J)$. 

Note that each choice of Reeb vector field $\xi\in\gt^+(\cald,J)$ gives rise to an infinite dimensional contractible space $\calS(M,\xi)$ of Sasakian structures \cite{BG05}, and we often have need to obtain a particular element of $\calS(M,\xi)$ by deforming the contact structure $\cald\mapsto \cald_\varphi$ by a contact isotopy $\eta\mapsto \eta +d^c\varphi$ where $\varphi\in C^\infty(M)^\bbt$ is a smooth function invariant under the torus $\bbt$. We note that the Sasaki cone $\gt^+(\cald,J)$ is invariant under such contact isotopies in the sense that $\gt^+(\cald_\varphi,J_\varphi)\approx \gt^+(\cald,J)$. We shall often make such a choice $\cals=(\xi,\eta,\Phi,g)\in\calS(M,\xi)$ and identify it with the element $\xi\in\gt^+(\cald,J)$.

\begin{remark}\label{torrem}
When $\dim\gt^+(\cald,J)=n+1$ we have what in \cite{BG00b} was called a toric contact manifold of Reeb type. This is actually a toric K-contact manifold, and in \cite{BG00b} a Delzant \cite{Del88} type theorem was proved, that is, any toric K-contact manifold is Sasaki. Moreover, as in the symplectic case there is a strong connection between the geometry and topology of $(M,\cals)$ and the combinatorics of $\gt^+(\cald,J)$ \cite{BM93,Ler02a,Ler04,Leg10,Leg16}\footnote{The combinatorics studied in these references is that of the moment cone which is dual to the Sasaki cone $\gt^+(\cald,J)$.}. Much can also be said in the complexity 1 case ($\dim\gt^+(\cald,J)=n$) \cite{AlHa06}. 
\end{remark}

It is important to realize that there are two types of Reeb orbits, those that are closed (i.e periodic orbits) and those that are not. On a closed K-contact manifold a Reeb vector field in the Sasaki cone $\gt^+$ is $C^\infty$-close to a Reeb vector field all of whose orbits are periodic. What can one say about Reeb vector fields in the complement of $\gt^+$?
The famous Weinstein conjecture says that every Reeb vector field on a compact contact manifold has a periodic orbit, and this is known to hold on a compact simply connected K-contact manifold \cite{Ban90}. See also \cite{Gin96,AGH18}.

 We end this section with the following observation that applies to our examples.

\begin{proposition}\label{S3bun}
Let $\bbc\bbp^1\rightarrow S_\bfn \rightarrow N$ be a projective bundle where $N$ is a smooth projective algebraic variety of complex dimension $d_N\geq 2$, and let $M^{2d_N+3}$ be the total space of a Sasaki $S^1$ bundle over $S_\bfn$. Then $M^{2d_N+3}$ is a nontrivial lens space bundle (with fiber $F$) over $N$. Furthermore, $F=S^3$ if and only if the natural induced map $\pi_2(M)\lra\pi_2(N)$ is an epimorphism, and the natural induced map $\pi_1(M)\lra\pi_1(N)$ is a monomorphism. In particular, if $N$ is simply connected there is a choice of K\"ahler class on $S_\bfn$ such $F=S^3$.
\end{proposition}

\begin{proof}
By composition we have a smooth bundle $F\rightarrow M^{2d_N+3}\rightarrow N$, and by construction the $S^1$ action on $M^{2d_N+3}$ only acts on the fibers $F$. Moreover, since the total space of this bundle is Sasaki, the bundle is nontrivial. So its restriction to $F$ is also nontrivial. It follows that $F$ is a lens space and we have the commutative diagram
\begin{equation}\label{commdiag2}
\begin{matrix}
S^1 &\lra &F& \lra &\bbc\bbp^1 \\
\decdnar{id}&&\decdnar{}&&\decdnar{}\\
S^1 &\lra &M^{2d_N+3}& \lra &S_\bfn \\
&&\decdnar{}&&\decdnar{}\\
&&N &\fract{id}{\lra} & N.
\end{matrix}
\end{equation}
Now since $N$ is K\"ahler its third Betti number $b_3(N)$ is even. Furthermore, since $M^{2d_N+3}$ is Sasaki of dimension at least 7, its third Betti number $b_3(M)$ must also be even which implies that the Euler class of the lens space bundle cannot vanish implying that the bundle is nontrivial.

Since $F$ is a lens space, $ \pi_2(F)=0$, so the long exact homotopy sequence becomes
\begin{equation}\label{homexseq}
\BOne\lra \pi_2(M)\lra\pi_2(N)\lra \pi_1(F)\lra\pi_1(M)\lra \pi_1(N)\lra \BOne.
\end{equation}
So when the induced map $\pi_2(M)\lra\pi_2(N)$ is an epimorphism, and the induced map $\pi_1(M)\lra \pi_1(N)$ is a monomorphism we have $ \pi_1(F)=\BOne$ which gives $F=S^3$ in this case. The converse is also clear from the homotopy exact sequence.

Now if $N$ is simply connected so is $S_\bfn$. Thus, by choosing a primitive K\"ahler class on $S_\bfn$, we can take $M^{2d+3}$ to be simply connected. Furthermore, we can choose the K\"ahler class on $S_\bfn$ such that its restriction to $\bbc\bbp^1$ is primitive. It then follows that $F=S^3$. 
\end{proof}

\section{Yamazaki's Fiber Join}\label{Yamsect}
Yamazaki \cite{Yam99} constructed his fiber join in the category of regular K-contact manifolds which as shown in \cite{BoTo20b} restricts to the Sasakian case in a natural way. We refer to ob.cit for details. Here we briefly recall that the fiber join is constructed by considering $d+1$ regular Sasaki manifolds $M_j$ over a smooth algebraic variety $N$ with $d+1$ K\"ahler forms $\gro_j$ on $N$ that are not necessarily distinct. One then constructs a smooth manifold $M=M_1\star_f\cdots\star_f M_{d+1}$ as the unit sphere in the complex vector bundle $E=\oplus_{j=1}^{d+1}L_j^*$ where $L_j$ denotes the complex line bundle on $N$ associated to $M_j$ such that $c_1(L_j)=[\gro_j]$ and $L_j^*$ is its dual. We shall refer to such a fiber join as a {\it Sasaki-Yamazaki fiber join}.  Topologically, we have

\begin{proposition}\label{fjointopprop}
Let $M$ be a Sasaki-Yamazaki fiber join as described above. Then
\begin{enumerate}
\item $M$ is a $S^{2d+1}$ bundle over $N$ with a $d+1$ dimensional Sasaki cone. Moreover,
\item if $d\geq n$ then $M$ has the cohomology groups of the product $S^{2d+1}\times N$; whereas,
\item if $d<n$ then the Euler class of the bundle does not vanish, and the Betti numbers satisfy $b_{2d+2i}(M)=b_{2d+2i}(N)-1$ where $i=1,\ldots,n-d$. 
\end{enumerate}
\end{proposition}

\begin{proof}
That $M$ is an $S^{2d+1}$ bundle follows from the construction, and Theorem 3.4 in \cite{BoTo20b} shows that $M$ admits a $d+1$ dimensional family of Sasakian structures.
When $d\geq n$ the Euler class of the bundle vanishes and the Leray-Serre spectral sequence collapses giving the product groups in the limit. However, if $d<n$ with $M$ having a Sasakian structure, the odd Betti numbers less than half the dimension must be even (cf. \cite{BG05}). Moreover, the odd Betti numbers of $N$ are also even, and the even Betti numbers are greater than zero. So if the Euler class vanishes the orientation class $\gra$ of the sphere which lies in the $E^{0,2d+1}_2$ term of the spectral sequence would survive to infinity which would imply that the Betti number $b_{2d+1}$ is odd. This contradicts the fact that $M$ has a Sasakian structure since $2d+1<2n< \frac{1}{2}\dim~M$. Thus, the Euler class,  which is represented by the differential $d_{2d+2}(\gra)$, cannot vanish in this case. So the real class $d_{2d+2}(\gra)\in E^{2d+2,0}_\infty$ is killed which reduces the $(2d+2)th$ Betti number by one. The other equalities follow from this and naturality of the differential.
\end{proof}

The Euler class of the bundle is $\gro^{d+1}$ where $\gro$ is an integral K\"ahler form on $N$. We want to determine the conditions under which a sphere bundle is a fiber join. It is convenient to think of this in terms of $G$-structures. An oriented $S^{2d+1}$-bundle over $N$ is an associated bundle to a principal bundle with group $SO(2d+2)$. 

\begin{proposition}\label{sphbunfibj}
An $S^{2d+1}$-bundle $S(E)$ over a smooth projective algebraic variety $N$ is of the form $S(\oplus_iL_i)$
if and only if the group of the corresponding principal bundle is the maximal torus $\bbt^{d+1}_\bbc$. Moreover, this is a Sasaki-Yamazaki fiber join if there is a choice of complex line bundles $L_i$ such that $c_1(L_i^*)$ is positive definite for all $i=1,\ldots,d+1$.
\end{proposition}

\begin{proof}
The only if part is clear.
Conversely, let $M$ be the total space of the unit sphere bundle in a complex vector bundle $E$ over a smooth projective algebraic variety $N$. Assume that the structure group of $E$ reduces to a maximal torus $\bbt^{d+1}_\bbc$. Then $E$ is isomorphic to a sum of complex line bundles $\oplus_{i=1}^{d+1}L_i$. Assume further that the $L_i$ can be chosen such that $c_1(L_i^*)$ is positive definite for $i=1,\ldots,d+1$. But this gives precisely the fiber join of the corresponding $S^1$ bundles over $N$.
\end{proof}

Let $M$ be a Sasaki-Yamazaki fiber join. Then as discussed above $M$ is an $S^{2d+1}$ bundle over a smooth projective algebraic variety $N$ for some $d\geq 1$.  The Sasakian structure on $M$ restricts to the standard weighted Sasakian structure on each fiber $S^{2d+1}$. When the weights are integers, it is convenient to describe this by the following commutative diagram of $S^1$ actions labelled by a weight vector $\bfw$:
\begin{equation}\label{bundiag}
\xymatrix{
S^1_\bfw \ar[d]^{id} \ar[r]  &S^{2d+1}\ar[d] \ar[r] &\bbc\bbp^d[\bfw] \ar[d]  \\
               S^1_\bfw \ar[r]&M_\gw \ar[r] \ar[d] &\bbp_\bfw(\oplus_{j=1}^{d+1}L^*_j)\ar[d]  \\
&N  \ar[r]^{id}  &N.}
\end{equation}

\subsection{Quasi-regular Quotients when $d=1$}\label{qrquotients}

For the case $d=1$ and co-prime $\bfw = (w_1,w_2) \in (\bbz^+)^2$, we want to understand 
$\bbp_\bfw(\oplus_{j=1}^{d+1}L^*_j) $ in the diagram \eqref{bundiag}.
To this end we will follow the ideas in Section 3.6 of \cite{BoTo13}.

Let $M_i^3 \rightarrow N$ denote the primitive principal $S^1$-bundle corresponding to the line bundle $L_i$. Here we assume that $N$ is a smooth projective algebraic manifold. So $c_1(L_i^*)$ equals some (negative) integer $d_i$ times
a primitive cohomology class that in turns defines $M_i^3$. [Recall that $L_i$ has to be a positive line bundle over $N$.]
Consider the $S^1 \times S^1 \times \bbc^*$ action $\cala_{\bfw, L_1,L_2}$ on $M_1^3\times M_2^3 \times \bbc^2$ defined by
\begin{equation}
\cala_{\bfw, L_1,L_2}(\lambda_1,\lambda_2,\tau)(x_1,u_1,x_2,u_2;z_1,z_2)=(x_1, \lambda_1 u_1, x_2, \lambda_2 u_2; \tau^{w_1} \lambda_1^{d_1} z_1, \tau^{w_2} \lambda_2^{d_2}z_2),
\end{equation}
where $\lambda_1,\lambda_2,\tau\in \bbc^*$ and $|\lambda_i|=1$.
Then $\bbp_\bfw(L^*_1\oplus L^*_2)$ should equal 
$$M_1^3\times M_2^3 \times \bbc^2/\cala_{\bfw, L_1,L_2}(\lambda_1,\lambda_2,\tau).$$

Now, we also can define a $w_1w_2$-fold covering map $\tilde{h}_\bfw: M_1^3\times M_2^3 \times \bbc^2 \rightarrow M_1^3\times M_2^3 \times \bbc^2$ by
$$\tilde{h}(x_1,u_1,x_2,u_2;z_1,z_2)= (x_1,u_1,x_2,u_2;z_1^{w_2},z_2^{w_1})$$
and this gives a commutative diagram

\begin{equation}
\xymatrix{M_1^3\times M_2^3 \times \bbc^2 \ar[d]^{\tilde{h}_\bfw}& \xrightarrow{\cala_{\bfw, L_1,L_2}\left(\lambda_1,\lambda_2,\tau\right)} & M_1^3\times M_2^3 \times \bbc^2 \ar[d]^{\tilde{h}_\bfw}\\
M_1^3\times M_2^3 \times \bbc^2& \xrightarrow{\cala_{{\bf 1}, L_1^{w_2},L_2^{w_1}}\left(\lambda_1,\lambda_2,\tau^{w_1w_2}\right)}& M_1^3\times M_2^3 \times \bbc^2}
\end{equation}
and so we have a fiber preserving biholomorphism $h_\bfw: \bbp_\bfw(L^*_1\oplus L^*_2) \rightarrow \bbp((L^*_1)^{w_2}\oplus (L^*_2)^{w_1})$
and we can write $\bbp_\bfw(L^*_1\oplus L^*_2) $ as the log pair
$(\bbp((L^*_1)^{w_2}\oplus (L^*_2)^{w_1}),\Delta_\bfw)$, where
$\Delta_\bfw = (1-1/w_1)D_1 + (1-1/w_2)D_2$ and $D_1,D_2$ are the zero and infinity sections, respectively, of the bundle $\bbp((L^*_1)^{w_2}\oplus (L^*_2)^{w_1}) \rightarrow N$.

\begin{remark}
Note that if $\bfw = (1,1)$ this checks out with the usual regular quotient. If the principal bundles $M_1^3$ and $M_2^3$ are equal, we can choose $(w_1,w_2)=(d_1,d_2)/a$ with $a=-\gcd(|d_1|,|d_2|)$ to get
that  $(L^*_1)^{w_2}=(L^*_2)^{w_1}$ and so $\bbp((L^*_1)^{w_2}\oplus (L^*_2)^{w_1})$ is trivial and the quasi-regular quotient is a product as expected from Proposition 3.8 of \cite{BoTo20b}
\end{remark}

By utilizing the set-up in Section A.3 of \cite{BoTo20b} we can also determine the quasi-regular K\"ahler class (up to scale) in the case with $d=1$ and co-prime $\bfw = (w_1,w_2) \in (\bbz^+)^2$ as above.
Indeed, from (9) of \cite{BoTo20b} we have that

\begin{equation}\label{quasiKahlerClass}
w_1 w_2 d\eta_\bfw= w_2(r_1^2d\eta_1 + 2(r_1dr_1\wedge (\eta_1+d\theta_1)))+w_1(r_2^2d\eta_2 + 2(r_2dr_2\wedge (\eta_2+d\theta_2))),
\end{equation}
where $(r_j,\theta_j)$ denote the polar coordinates on the fiber of the line bundle $L^*_j$ (chosen via a Hermitian metric on the line bundle).

As explained in Section A.3 of \cite{BoTo20b}, we can say that $z_0:= \frac{1}{2}r_1^2$ and $z_\infty:= \frac{1}{2}r_2^2$  are the moment maps of the natural $S^1$ action on $L^*_1$ and $L^*_2$, respectively. 
On $2=z_0+z_\infty$, the function 
$z:= z_0-1=1-z_\infty$ descends to a fiberwise moment map (with range [-1,1]) for the induced $S^1$ action on $\bbp(\BOne \oplus (L_1)^{w_2}\otimes (L^*_2)^{w_1}) \rightarrow N$.
Using that $r_1^2=2(z+1)$, $r_2^2=2(1-z)$, $r_1\,dr_1=dz$, and $r_2\,dr_2=-dz$, we rewrite \eqref{quasiKahlerClass} to
$$w_1 w_2 d\eta_\bfw= 2(w_2 d\eta_1+w_1 d\eta_2) + 2d(z\theta),$$
where $\theta := w_2(\eta_1+d\theta_1)-w_1(\eta_2+d\theta_2)$ is a connection form on $(L_1)^{w_2}\otimes (L^*_2)^{w_1}$.
Now this descends to a K\"ahler form on $\bbp((L^*_1)^{w_2}\oplus (L^*_2)^{w_1})=\bbp(\BOne \oplus (L_1)^{w_2}\otimes (L^*_2)^{w_1}) \rightarrow N$ with
K\"ahler class $2(2\pi(w_2[\omega_1] + w_1[\omega_2]) + \Xi)$ where $c_1(L_j)=[\omega_j]$ and $\Xi/(2\pi)$ is the Poincare dual of $(D_1+D_2)$.

We can summarize our findings for $d=1$ in the following proposition.

\begin{proposition}\label{d=1quasiregular}
For $d=1$ and co-prime $\bfw = (w_1,w_2) \in (\bbz^+)^2$, the quasi-regular quotient of $M_\gw$ with respect to $\xi_\bfw$ is
the log pair $B_{\gw, \bfw}:=(\bbp((L^*_1)^{w_2}\oplus (L^*_2)^{w_1}),\Delta_\bfw)$, where
$\Delta_\bfw = (1-1/w_1)D_1 + (1-1/w_2)D_2$ and $D_1,D_2$ are the zero and infinity sections, respectively, of the bundle $\bbp((L^*_1)^{w_2}\oplus (L^*_2)^{w_1}) \rightarrow N$.
Moreover, up to scale, the induced transverse K\"ahler class on $B_{\gw, \bfw}$ is equal to $2\pi(w_2[\omega_1] + w_1[\omega_2]) + \Xi$ where $c_1(L_j)=[\omega_j]$ and $\Xi/(2\pi)$ 
is the Poincare dual of $(D_1+D_2)$.
\end{proposition}

\begin{remark}\label{colinearsanitycheck}
We can do the following sanity check: If colinearity (see \cite{BoTo20b} for the definition) holds on top of the above assumptions, we have according to Proposition 15 of \cite{BoTo20b} that  the fiber join is just a regular $S^3_{\tilde{\bfw}}$-join as in \cite{BoTo14a}.
Here $\omega_i=b_i \omega_N$ for $[\omega_N]$ a primitive integer K\"ahler class. 
Connecting with the notation in \cite{BoTo14a} (setting $w_i$ from \cite{BoTo14a} equal to $\tilde{w}_i$) we have $l_1(\tilde{w}_1,\tilde{w}_2)=(b_1,b_2)$ and $l_2=1$.
Now Proposition \ref{d=1quasiregular} is consistent with Theorem 3.8 of \cite{BoTo14a} (with $v_i=w_i$) saying that
the quotient of $\xi_\bfw$ has $n= b_1w_2-b_2w_1$. Moreover, the transverse K\"ahler class is then
$$2\pi (w_2[\omega_1]+w_1[\omega_2]) + \Xi = 2\pi (b_1w_2+b_2w_1) [\omega_N] +\Xi = \frac{1}{r} [\omega_{N_n}] + \Xi$$ 
with $r_a=\frac{b_1w_2-b_2w_1}{b_1w_2+b_2w_1}$ and
$[\omega_{N_n}]:=2\pi n [\omega_N]=c_1((L_1)^{w_2}\otimes (L^*_2)^{w_1})$.
This is consistent with (44) and (59) in \cite{BoTo14a}. 
\end{remark}

\subsection{The General $d$ Case}\label{gendsec}
For the fiber join $M_\gw $ we have in particular that the complex manifold arising as the quotient of the regular Reeb vector field $\xi_1$ 
is equal to $\bbp\left(\oplus_{j=1}^{d+1} L^*_j\right) \rightarrow N$. Recall from \cite{BoTo20b} that this is an {\em admissible} projective bundle as defined in 
\cite{ACGT08} exactly when the following all hold true:
\begin{enumerate}
\item The base $N$ is a local product of K\"ahler manifolds $(N_a,\Omega_{N_a})$, $a \in \cala \subset \bbn$, where $\cala$ is a finite index set.
This means that there exist simply connected K\"ahler manifolds $N_{a}$ of complex dimension $d_a$ such that $N$ is covered by $\prod_{a\in \cala}N_a$. On each $N_a$ there is an $(1,1)$ form
$\Omega_{N_a}$, which is a pull-back of a tensor (also denoted by $\Omega_{N_a}$) on $N$, such that $\Omega_{N_a}$ is a K\"ahler form of a constant scalar curvature K\"ahler (CSCK) metric $g_a$. 
\item There exist $d_0, d_\infty \in \bbn \cup \{0\}$, with $d=d_0+d_\infty +1$, such that $E_0:= \oplus_{j=1}^{d_0+1} L^*_j$ and
$E_\infty := \oplus_{j=d_0+2}^{d_0+d_\infty +2} L^*_j$ are both projectively flat hermitian holomorphic vector bundles.
{\em This would, for example, be true if $L^*_j= L_0$ for $j=1,...,d_0+1$ and $L^*_j= L_\infty$ for $j=d_0+2,...,d_0+d_\infty+2$, where $L_0$ and $L_\infty$ are some holomorphic line bundles. That is, $E_0=L_0\otimes \bbc^{d_0+1}$ and $E_\infty = L_\infty\otimes\bbc^{d_\infty+1}$. 
More generally,
$c_1(L^*_1) =  \cdots = c_1(L^*_{d_0+1})$ and $c_1(L^*_{d_0+2})= \cdots = c_1(L^*_{d_0+d_\infty+2})$
would be sufficient.}
\item $\frac{c_1(E_\infty)}{d_\infty+1}-\frac{c_1(E_0)}{d_0+1}= \sum_{a\in \cala}  [\epsilon_a\Omega_{N_a}]$, where $\epsilon_a=\pm 1$.
\end{enumerate}
The K\"ahler cone of the total space of an admissible bundle $\bbp\left( E_0\oplus E_\infty\right) \rightarrow N$ has a subcone of so-called {\bf admissible K\"ahler classes} (defined in Section 1.3 of \cite{ACGT08}). This subcone has dimension $|\cala|+1$ and, in general, this is not the entire K\"ahler cone. However, by Remark 2 in \cite{ACGT08}, if $b_2(N_a)=1$ for all $a\in \cala$ and $b_1(N_a) \neq 0$ for at most one $a\in \cala$, then the entire K\"ahler cone is indeed admissible.

\subsection{Admissibility}
As briefly discussed in \cite{BoTo20b}, it is convenient to have refined notions of admissibility.

\begin{definition}\label{stradmdef}
Any fiber join $M_\gw$ where the quotient of the regular Reeb vector field $\xi_1$ 
is an admissible projective bundle will also be called {\bf admissible}.
If further the transverse K\"ahler class of the regular quotient is a pullback of an admissible K\"ahler class, then we call $M_\gw$ {\bf strongly admissible}.
\end{definition}

\begin{remark}
Note that in Definition 4.1 of \cite{BoTo20b} we introduced the condition of being {\bf super admissible}. There we required the entire K\"ahler cone of the regular admissible quotient to be admissible. Of course, if that is the case then in particular the transverse K\"ahler class of the regular quotient is a pullback of an admissible K\"ahler classes. Thus $M_\gw$ is strongly admissible if it is super admissible. In fact we have

\begin{proposition}\label{properincl}
Generally the inclusions
$$ {\rm super~ admissible} \subset {\rm strongly~ admissible} \subset {\rm admissible}$$
are proper.
\end{proposition}
\end{remark}

The proof of this proposition is a consequence of either of the Examples \ref{admnotstrong} or \ref{2ndex} below.

\begin{example}\label{admnotstrong}
Let $\Sigma_g$ be a Riemann surface of genus $g>1$ and let $\omega_{\Sigma_g}$ denote the unit area K\"ahler form of the constant scalar curvature K\"ahler metric on $\Sigma_g$.
Now consider $N=\Sigma_g\times\Sigma_g$ (i.e. $N_1=N_2=\Sigma_g$) and let $\pi_a$ denote the projection from $N$ to the $a^{th}$ factor. Then $\gamma_a:=[\pi_a^*\omega_{\Sigma_g}]\in H^2(N,\bbz)$.
Let $\delta \in H^2(N,\bbz)$ denote the Poincar\'e dual of the diagonal divisor in $N$ defined by the diagonal curve $\{(x,x)\,|\,x\in \Sigma_g\}$. Then from Theorem 3.1 of \cite{Shu} (which uses Nakai's criterion for ample divisors), we
know that $l_s:=(s-1)(\gamma_1+\gamma_2)+\delta \in H^2(N,\bbz)$ is in the K\"ahler cone of $N$ if and only if $s>g$.

Now we form a $d=1$ Yamazaki fiber join by choosing line bundles $L_1$ and $L_2$ over $N$ such that
$c_1(L_1)=[\omega_1]=l_{g+2}$ and $c_1(L_2)=[\omega_2]=l_{g+1}$. In the above setting $L_1^*=E_0$ and $L_2^*=E_\infty$ and we easily see that the fiber join is indeed admissible with
 $c_1(E_\infty)-c_1(E_0)=c_1(L_1)-c_1(L_2)=l_{g+2}-l_{g+1}=\gamma_1+\gamma_2$. Specifically, the regular quotient equals the admissible bundle
 $$S_g:=\bbp\left(L_1^*\oplus L_2^*\right) \rightarrow N=\bbp\left(\BOne \oplus L_1\otimes L_2^*\right) \rightarrow N=
\bbp\left(\BOne \oplus \calo(1,1)\right) \rightarrow \Sigma_g\times\Sigma_g.$$ Note that with the above notation
 $[\Omega_{N_a}]=\gamma_a$ (and $\epsilon_a=+1$). 
On $S_g$, the admissible K\"ahler classes are up to scale of the form $$\frac{1}{x_1}[\Omega_{N_1}]+\frac{1}{x_2}[\Omega_{N_2}]+\Xi,$$
where $0<x_a<1$ (following Section 1.3 of \cite{ACGT08}). According to Proposition \ref{d=1quasiregular}, the regular transverse K\"ahler class is, up to scale, the pull-back
 of $2\pi([\omega_1] + [\omega_2]) + \Xi$. This equals 
 $$2\pi(l_{g+2}+l_{g+1})+\Xi=2\pi((2g+1)(\gamma_1+\gamma_2)+2\delta)+\Xi=2\pi\bigl((2g+1)[\Omega_{N_1}]+(2g+1) [\Omega_{N_2}]+2\delta\bigr)+\Xi$$
 which due to the ``$2\delta$'' bit is not an admissible K\"ahler class. Therefore, the fiber join is not strongly admissible.
 
 Furthermore, it is possible to chose the line bundles $L_1$ and $L_2$ so that the fiber join is strongly admissible (cf. Section 5), but it will never be super admissible due to the fact that the K\"ahler cone of $N$ consist of more than just product classes and thus there are non-admissible K\"ahler classes on the total space of the $\bbc\bbp^1$-bundle of the regular quotient. Hence, the inclusions in Proposition \ref{properincl} are proper.
 \end{example}

 \begin{example}\label{2ndex}
Another example of admissible but not strongly admissible is the following case. Let $N=\bbp(\BOne \oplus \calo(1,-1)) \rightarrow \bbc\bbp^1\times\bbc\bbp^1$.
Let $\Omega_{FS}$ denote the standard Fubini-Study K\"ahler form on $\bbc\bbp^1$, let $\pi_i$ denote the projection from $N$ to the $i^{th}$ factor in the product
$\bbc\bbp^1\times\bbc\bbp^1$, and let $\chi$ denote the Poincar\'e dual of $2\pi(D^N_1+D^N_2)$, where $D^N_1$, $D^N_2$ are the zero and infinity sections of
$N \rightarrow \bbc\bbp^1\times\bbc\bbp^1$.
Now consider the two CSC K\"ahler forms $\omega_1$ and $\omega_2$ on $N$ with K\"ahler classes
$$[\omega_1]=2\left(3[\pi_1^*\Omega_{FS}] +3[\pi_2^*\Omega_{FS}] +\frac{\chi}{2\pi}\right)$$
and
$$[\omega_2]=2[\pi_1^*\Omega_{FS}] +2[\pi_2^*\Omega_{FS}] +\frac{\chi}{2\pi},$$ 
respectively.
(See e.g. Theorem 9 in \cite{ACGT08} to confirm that $[\omega_1]$ and $[\omega_2]$ are indeed represented by CSC K\"ahler forms.)

Now we form a $d=1$ Yamazaki fiber join by choosing line bundles $L_1$ and $L_2$ over $N$ such that
$c_1(L_1)=[\omega_1]$ and $c_1(L_2)=[\omega_2]$. In the above setting $L_1^*=E_0$ and $L_2^*=E_\infty$ and we easily see that the fiber join is indeed admissible with
$c_1(E_\infty)-c_1(E_0)=c_1(L_1)-c_1(L_2)=4[\pi_1^*\Omega_{FS}] +4[\pi_2^*\Omega_{FS}] +\frac{\chi}{2\pi}$. Specifically, the regular quotient equals the admissible bundle
 $$S:\bbp\left(\BOne \oplus L\right) \rightarrow N$$ such that 
 $c_1(L)= [\Omega_{N}]:=4[\pi_1^*\Omega_{FS}] +4[\pi_2^*\Omega_{FS}] +\frac{\chi}{2\pi}$ and
$\Omega_N$ is a CSC K\"ahler form on $N$. Note that $S$ is a so-called {\em stage four Bott manifold} given by the matrix
$$
A=\begin{pmatrix}
1&0&0&0\\
0&1&0&0\\
1&-1&1&0\\
5&3&2&1
\end{pmatrix}.$$
[See e.g. Section 1 of \cite{BoCaTo17} for details.] It is important to note that the CSC K\"ahler manifold $(N,\Omega_N)$ is irreducible in the sense that
for (1) at the beginning of Subsection \ref{gendsec}, $\cala$ must be just $\{1\}$.

Following Section 1.3. in \cite{ACGT08}, we have that on $S$, the admissible K\"ahler classes are up to scale of the form 
$$\frac{2\pi}{x}[\Omega_{N}]+\Xi,$$
where $0<x<1$, $\Xi$ denote the Poincare dual of $2\pi(D_1+D_2)$, and  $D_1,D_2$ are the zero and infinity sections, 
respectively, of the bundle $\bbp(\BOne \oplus L) \rightarrow N$.  

According to Proposition \ref{d=1quasiregular}, the regular transverse K\"ahler class is, up to scale, the pull-back
 of $2\pi([\omega_1] + [\omega_2]) + \Xi$. This equals 
 $$2\pi(8[\pi_1^*\Omega_{FS}] +8[\pi_2^*\Omega_{FS}]+3 \frac{\chi}{2\pi}) + \Xi$$
which cannot be written as (the rescale of) $\frac{2\pi}{x}[\Omega_{N}]+\Xi$ for any $0<x<1$.
Thus this is not an admissible K\"ahler class and therefore the fiber join is not strongly admissible.
\end{example}

\subsection{The Main Theorems}\label{mainthm}
For Theorems \ref{1stroundexistence} and \ref{2ndroundexistence} below, we only need the strongly admissible condition. In \cite{BoTo20b} we used the above observations together with existence results in \cite{Gua95}, \cite{Hwa94}, and \cite{HwaSi02} (specifically, the slight generalization in the form of Propostion 11 of \cite{ACGT08}) to prove the following theorem:

\begin{theorem}[\cite{BoTo20b}]\label{1stroundexistence}
Let $M_\gw$ be a strongly admissible fiber join whose regular quotient is a ruled manifold of the form $\bbp(E_0\oplus E_\infty)\longrightarrow N$ where $E_0 ,E_\infty$ are projectively flat hermitian holomorphic vector bundles on $N$ of complex dimension $(d_0+1),(d_\infty+1)$ respectively, and $N$ is a local K\"ahler product of non-negative CSC metrics. Then the Sasaki cone of $M_\gw$ has an open set of extremal Sasaki metrics (up to isotopy).
\end{theorem}

Together with E. Legendre and H. Huang we recently obtained the following result on admissible K\"ahler manifolds:

\begin{theorem}[Theorem 3.1 in \cite{BHLT22}]\label{bhlt22}
Suppose $\Omega$ is a rational admissible K\"ahler class on the admissible manifold $N^{ad}=\bbp(E_0 \oplus E_{\infty}) \lra N$, where $N$ is a compact K\"ahler manifold which is a local product of nonnegative CSCK metrics. 
Let $(M,\cals)$ be the Boothby-Wang constructed  Sasaki manifold given by an appropriate rescale of $\Omega$. Then the corresponding Sasaki-Reeb cone will always have a
(possibly irregular) CSC-ray (up to isotopy).
\end{theorem}

The proof of this theorem (Section 3.1 of \cite{BHLT22}) reveals that this CSC Sasaki metric lies in a 2-dimensional subcone of $\gt^+(\cals)$ which is exhausted by extremal Sasaki metrics.
Further, since this subcone is constructed via Killing potentials coming from a moment map induced by a fiber wise $S^1$-action on the admissible bundle, it is clear that
this is also a subcone of $\gt^+_{sph}$ (and all of $\gt^+_{sph}$ when $d=1$). Recall from Section 2.2 of \cite{BoTo20b} that $\gt^+_{sph}$ is defined to be the natural $(d+1)$-subcone of the Sasaki-Reeb cone of $M_\gw$ coming from considering the standard Sasaki CR structure on $S^{2d+1}$.
In light of all this, we can thus easily improve Theorem \ref{1stroundexistence} to give Theorem \ref{2ndroundexistenceintro} in the Introduction, namely

\begin{theorem}\label{2ndroundexistence}
Let $M_\gw$ be a strongly admissible Yamazaki fiber join whose regular quotient is a ruled manifold of the form $\bbp(E_0\oplus E_\infty)\longrightarrow N$ where $E_0 ,E_\infty$ are projectively flat hermitian holomorphic vector bundles on $N$ of complex dimension $(d_0+1),(d_\infty+1)$ respectively, and $N$ is a local K\"ahler product of non-negative CSC metrics. Then $\gt^+_{sph}$ has a $2$-dimensional subcone of extremal Sasaki metrics (up to isotopy) which contains at least one ray of CSC Sasaki metrics.
\end{theorem}

\section{Further Examples}\label{Exsect}
In this section we work out the details of examples of fiber joins in dimensions 5 and 7. We consider only the case with $d=1$, i.e. $d_0=d_\infty=0$. So we have an $S^3$ bundle $M$, which we shall assume to be strongly admissible, over a smooth projective algebraic variety $N$. We begin with the simplest case, namely where $N$ is a Riemann surface, so the simplest fiber join is of dimension 5. Even in this case the geometry is quite involved. Note that the genus $g=0$ case is a straightforward special case of Theorem \ref{2ndroundexistence} whose Sasaki cone is strictly larger than $\gt^+_{sph}$; hence, we concentrate on $g\geq 1$.

In this case the fiber $\bbc\bbp^d[\bfw]$ is the log pair $(\bbc\bbp^1,\grD_\bfw)$ with branch divisors
$$\grD_\bfw=\bigl(1-\frac{1}{w_1}\bigr)D_1 +\bigl(1-\frac{1}{w_2}\bigr)D_2.$$
Here we have $c_1(L_\infty)-c_1(L_0)=\sum_a[\gre_a\grO_a]$. In order to construct a non-colinear fiber join of this kind we must have the Picard number $\grr(N)\geq 2$. 
In this case we may see the rays determined by $\xi_\bfw$ explicitly as $CR$-twists of the regular quotient \cite{ApCa18}. 
Indeed, following the notation of Section 3 of \cite{BHLT22}, on the regular quotient,
$N^{ad} =   \bbp\bigl({\BOne} \oplus (L_0^*\otimes L_\infty)\bigr) \rightarrow N$,
we have a moment map $\gz: N^{ad} \rightarrow [-1,1]$.
A choice of $c\in (-1,1)$ creates a new Sasaki structure (with Reeb vector field $\xi_c$) on $M_\gw$ via the lift of $f=c\gz+1$ from $S$ to $M_\gw$.
In turn, this lift may be identified with $c\,z+1$, where $z:= z_0-1=1-z_\infty$ is given in the discussion above Proposition \ref{d=1quasiregular}. In particular, $z_0$ and $z_\infty$  are the moment maps of the natural 
$S^1$ action on $L^*_1$ and $L^*_2$, respectively. Thus, the weighted combination, $w_1 z_0+w_2 z_\infty$, should define the Reeb vector field $\xi_\bfw$ and since
$$w_1 z_0+w_2 z_1 = (w_1-w_2) z +(w_1+w_2)=(w_1+w_2)(\frac{w_1-w_2}{w_1+w_2} z +1),$$
we see that (up to scale) $\xi_\bfw$ corresponds to choosing $c=\frac{w_1-w_2}{w_1+w_2} $ in the $CR$-twist.


\subsection{$N=\Sigma_g$, a compact Riemann surface of genus $g\geq1$.}
It is well known that if $N=\grS_g$, a Riemann surface of genus $g$, then an odd dimensional  sphere bundle $M$ over $N$ is diffeomorphic to the trivial bundle $S^{2d+1}\times \grS_g$ or the unique non-trivial bundle $S^{2d+1}\tilde{\times} \grS_g$ \cite{Stee51}. We will consider $d=1$ fiber joins over $N=\Sigma_g$. Since these are necessarily colinear, they have earlier been treated as $S^3_\bfw$ joins
\cite{BoTo13}, but not in the setting of Yamazaki fiber joins. 
Let $\omega_{\Sigma_g}$ denote the unit area K\"ahler form of the constant scalar curvature K\"ahler metric on $\Sigma_g$ and 
let $k_1>k_2>0$ be integers (the case $0<k_1< k_2$ is completely similar) and let $L_1,L_2$ be holomorphic line bundles over $\Sigma_g$ such that $c_1(L_i)=k_i [\omega_{\Sigma_g}]$. The corresponding 
$d=1$ Yamazaki fiber join, $M_\bfk=S(L_1^*\oplus L_2^*) \rightarrow \Sigma_g$ has regular quotient $S_\bfn =   \bbp\bigl({\BOne} \oplus \calo(k_1-k_2)\bigr) \rightarrow \Sigma_g$ and regular transverse K\"ahler class equal (up to scale) 
to the admissible K\"ahler class $2\pi(k_1+k_2)[\omega_{\Sigma_g}] + \Xi$, which we can write as
$\frac{1}{x}\left( 2\pi (k_1-k_2)[\omega_{\Sigma_g}] \right)+\Xi$ with $x=\frac{k_1-k_2}{k_1+k_2}$. [See Remark \ref{colinearsanitycheck}.]
Note that since $g\geq 1$, we have that the Sasaki cone equals the $2$-dimensional cone $\gt^+_{sph}$

We now follow Section 3 of \cite{BHLT22}. On the regular quotient,
$S_\bfn$, we have a moment map $\gz: S_\bfn  \rightarrow [-1,1]$.
A choice of $c\in (-1,1)$ creates a new Sasaki structure  (with Reeb vector field $\xi_c=f\,\xi_{\mathbf 1}$) on $M_\gw$ via the lift of $f=c\gz+1$ from $S_\bfn$ to $M_\gw$. From the discussion in the beginning of Section \ref{Exsect} we know that $c=\frac{k_1-k_2}{k_1+k_2}=x$ corresponds to the Reeb vector field of the $S^3_{\tilde{\bfw}}$ join, $M^5_{g,l,\tilde{\bfw}}=S^3_g\star_{l,1}S^3_{\tilde{\bfw}}$ from Section 3.2 of \cite{BoTo13} with $S^3_g$ being the Boothby-Wang constructed smooth Sasaki structure over $(\Sigma_g,[\omega_{\Sigma_g}])$,
$l=\gcd(k_1,k_2)$, and $(\tilde{w_1},\tilde{w_2})=(\frac{k_1}{l},\frac{k_2}{l})$. Since this Reeb vector field is extremal by construction, we know a priori  that 
the set of extremal Sasaki rays in the Sasaki cone, $\gt^+_{sph}$ is not empty.

Proposition 3.10 of \cite{BHLT22} tells us that the Reeb vector field determined - up to homothety - by $c\in (-1,1)$ (as explained in the beginning of Section \ref{Exsect}) is extremal (up to isotopy)
if and only if $F_c(\gz)>0$, for $-1<\gz <1$, where the polynomial
$F_c(\gz)$ is given as follows:

\noindent
Let $s=\frac{2(1-g)}{k_1-k_2}$,  $x=\frac{k_1-k_2}{k_1+k_2}$ and define
$$
\begin{array}{ccl}
\alpha_{r,-4} & = & \int_{-1}^1(ct+1)^{-4}t^r(1+xt)\,dt\\
\\
\alpha_{r,-5} & = & \int_{-1}^1(ct+1)^{-5}t^r(1+xt)\,dt\\
\\
\beta_{r,-3} & = & \int_{-1}^1(ct+1)^{-3}xst^r\,dt \\
\\
&+ & (-1)^r(1-c)^{-3}(1-x) +(1+c)^{-3}(1+x).
\end{array}
$$
Then, 
\begin{equation}\label{wextrpol5mnf}
F_c(\gz)=(c\gz+1)^3 \left[ \frac{2(1-x)}{(1-c)^3}(\gz+1) + \int_{-1}^\gz Q(t)(\gz-t)\,dt\right],
\end{equation}
where 
$$Q(t) = \frac{2xs}{(ct+1)^3} - \frac{(A_1 t+A_2)(1+xt)}{(ct+1)^5}$$
and $A_1$ and $A_2$ are the unique solutions to the linear system
\begin{equation}\label{wextrpol5mnf2}
\begin{array}{ccl}
\alpha_{1,-5}A_1+\alpha_{0,-5}A_2&=& 2\beta_{0,-3}\\
\\
\alpha_{2,-5}A_1+\alpha_{1,-5}A_2&=& 2\beta_{1,-3}.
\end{array}
\end{equation}
Further, if the positivity of $F_c(\gz)$ is satisfied, then the extremal Sasaki structure is CSC exactly when
\begin{equation}\label{scsc5mnf}
\alpha_{1,-4}\beta_{0,-3} - \alpha_{0,-4}\beta_{1,-3}=0
\end{equation}
is satisfied. The left hand side of \eqref{scsc5mnf} equals
$\frac{4h(c)}{3(1-c^2)^5}$, with polynomial \newline
$h(c)=x(sx-2) +(5+x^2-sx)c-x(6+s x)c^2-(1-sx-3x^2)c^3$ and $h(\pm 1)=\pm 4(1\mp x)^2$. Thus, since $h(c)$ is negative at $c=-1$ and positive at $c=1$,
 \eqref{scsc5mnf}  always has at least one solution $c\in (-1,1)$. 
  
 We calculate $F_c(\gz)$:
 $$F_c(\gz)=\frac{(k_1+k_2)^2(1-\gz^2)p(\gz)}{4((1 - c)^2 k_1^2 + (1 + c)^2 k_2^2 + 4 (1 - c^2) k_1 k_2)},$$
 where $p(\gz)$ is a polynomial of degree $2$ whose coefficients depend on $k_1,k_2,g$ and $c$, but is more conveniently written as
 $$
 \begin{array}{ccl}
 p(\gz)&=& c^2 s x+3 c^2 x^2-c^2-2 c s x^2+3 c x^3-7 c x+s x^3-4 x^2+6\\
 \\
 &+& 2 x \left(3 c^2 x^2-c^2-4 c x-x^2+3\right)\gz\\
 \\
 &+& (c-x) \left(-c s x+3 c x^2-c+s x^2-2 x\right)\gz^2,
 \end{array}
 $$
 where $s=\frac{2(1-g)}{k_1-k_2}$,  $x=\frac{k_1-k_2}{k_1+k_2}$. 
Clearly $F_c(\gz)>0$ for all $\gz \in (-1,1)$ exactly when  $p(\gz)>0$ for all $\gz \in (-1,1)$. We have arrived at

\begin{proposition}\label{fibjoinextre}
Consider the $d=1$ fiber join $S^3\lra M_\bfk \lra \grS_g$ over a Riemann surface $\grS_g$ of genus $g\geq1$ with its natural Sasakian structure $\cals_c$ as described above. Then $\cals_c$ is extremal (up to isotopy) if and only if $p(\gz)>0$ for all $\gz \in (-1,1)$.
\end{proposition}


Note that \newline
$p(-1)=\frac{8k_2((1 - c)^2 k_1^2 + (1 + c)^2 k_2^2 + 4 (1 - c^2) k_1 k_2)}{(k_1+k_2)^3}$ and
 $p(1)=\frac{8k_1((1 - c)^2 k_1^2 + (1 + c)^2 k_2^2 + 4 (1 - c^2) k_1 k_2)}{(k_1+k_2)^3}>0$, thus $p(\pm 1)>1$, so we see right away that when $c=x$, $p(\gz)$ (which is now of degree one) is positive for $-1<\gz <1$. This confirms our expectation from above that $\xi_c$ is extremal when $c= \frac{k_1-k_2}{k_1+k_2}=x$.
 
It is easy to check that for $g>\frac{31 k_1^2+14 k_1 k_2+k_1+k_2^2+k_2}{k_1+k_2}$ and $c=\frac{k_1}{k_1+k_2}$, $p(0)<0$. Thus we see that for any fixed choice of integers $k_1>k_2>0$, $\gt^+_{sph}$ is not exhausted by extremal rays when $g$ is very large. This is expected in light of Theorem 5.1 in \cite{BoTo13}. 

From \cite{BoTo13} we have the following results:
\begin{enumerate}
\item (Proposition 5.5 in \cite{BoTo13} combined with Theorem 3 in \cite{ApCaLe21}) There is a unique ray in $\gt^+_{sph}$ with a CSC Sasaki metric (up to isotopy).
\item (Proposition 5.10 in \cite{BoTo13}) If $g\leq 1+3k_2$ then every ray in $\gt^+_{sph}$ has an extremal Sasaki metric (up to isotopy).
In particular, this is true whenever $g\leq 4$.
\end{enumerate}

Statement (1) means that \eqref{scsc5mnf} has a unique solution $c\in (-1,1)$ (i.e. 
the cubic $h(c)$ above has a unique real root $c\in (-1,1)$) and for this unique solution,
$p(\gz)>0$ for all $\gz \in (-1,1)$. An easy way to see the uniqueness of the real root directly from the present setup is to
make a change of variable\footnote{Note that $b$ is exactly what $c$ is in (51) of \cite{BoTo13}. This follows from
    our discussion in Section \ref{Yamsect}.}
 $c=\phi(b)=\frac{1-b}{1+b}$ [$\phi: (0,+\infty) \rightarrow (-1,1)$]
Then $h(c)$ transforms to $\tilde{h}(b)$, where
$$\tilde{h}(b)= \frac{4}{(b+1)^3}\left( (1-x)^2+(1-x)(2+2x-sx)b -(1+x)(2(1-x)-sx)b^2-(1+x)^2b^3\right).$$
Since the polynomial coefficients of the cubic 
$$(1-x)^2+(1-x)(2+2x-sx)b -(1+x)(2(1-x)-sx)b^2-(1+x)^2b^3$$ 
change sign exactly once (recall $sx\leq 0$ and $0<x<1$), we have (using Descartes' rule of signs) exactly one positive
root $b\in (0,+\infty)$ (corresponding to a unique root $c\in(-1,1)$ of $h(c)$).
Then too see that this $c$ value (let us call it $\hat{c}$) satisfies that $p(\gz)>0$ for all $\gz \in (-1,1)$
we can first observe that since $h(x)=3x(1-x^2)^2\neq 0$, $\hat{c}\neq x$. With that settled we may (solve for $s$ in $h(\hat{c})=0$ and) write
$s=\frac{3 \hat{c}^3 x^2-\hat{c}^3-6 \hat{c}^2 x+\hat{c} x^2+5 \hat{c}-2 x}{(1-\hat{c}^2) x (\hat{c}-x)}$.
Substituting this into $p(\gz)$ (and using that $x=\frac{k_1-k_2}{k_1+k_2}$) gives us
$p(\gz)=\frac{4 ((1 - \hat{c})^2 k_1^2 + (1 + \hat{c})^2 k_2^2 + 4 (1 - \hat{c}^2) k_1 k_2)(1 + \hat{c} \gz) (1 - \hat{c} x -\hat{c}\gz+x\gz)}{(1- \hat{c}^2)(k_1+k_2)^2}$. Since $0<x<1$ and $-1<\hat{c}<1$, it easily follows that $p(\gz)>0$ for $-1<\gz <1$.

Similarly, statement (2) is (re)verified if we show that
for $g\leq 1+3k_2$, $p(\gz)>0$ for all $c,\gz \in (-1,1)$. This is done easily by writing $p(\gz)$ in a new variable $y$:
$\gz=\psi(y)=\frac{1-y}{1+y}$ ($0<y<+\infty$) along with using the above transformation $c=\phi(b)=\frac{1-b}{1+b}$.
After multiplying by $(1+b)^2(1+y)^2$, this results in a
polynomial in the two variables $b,y>0$. The coefficients of this polynomial are all non-negative (with some strictly positive) precisely when
$g\leq 1+3k_i$ for $i=1,2$. Since (we assumed without loss of generality that) $k_1>k_2$, this is manifested by $g\leq 1+3k_2$.

\begin{example}
Assume now that $k_2=1$ and $g=5$ or $g=6$. Thus $g\leq 1+3k_2$ is false and Statement (2) cannot be applied.
Nevertheless we shall see that 
positivity of $p(\gz)$ for $-1<\gz<1$ still holds for all $k_1>1$: 
With $g=5$, $k_2=1$, $c=\phi(b)=\frac{1-b}{1+b}$, and $\gz=\psi(y)=\frac{1-y}{1+y}$, $p(\gz)$ rewrites to
$$\frac{32 \left(b^2 k_1^2(k_1-y+y^2)+3 b k_1^2 y+4 b k_1^2+4 b k_1 y^2+11 b k_1 y+(3 k_1-4) y+k_1+y^2\right)}{(b+1)^2 (k_1+1)^3 (y+1)^2}.$$
Since $k_1\geq 2$ and $y,b>0$, it is easy to see that this is always positive.

With $g=6$, $k_2=1$, $c=\phi(b)=\frac{1-b}{1+b}$, and $\gz=\psi(y)=\frac{1-y}{1+y}$, $p(\gz)$ rewrites to
$$\frac{32 \left(b^2 k_1^2(k_1-2y+y^2)+3 b k_1^2 y+4 b k_1^2+4 b k_1 y^2+13 b k_1 y+(3 k_1-5) y+k_1+y^2\right)}{(b+1)^2 (k_1+1)^3 (y+1)^2}$$
Since $k_1\geq 2$ and $y,b>0$, we see also in this case that this is always positive.

On the other hand, for $k_2\geq 2$, then $1+3k_2\geq 7$ and since $7$ is larger than both $5$ and $6$, we already know from Statement (2) above that positivity of $p(\gz)$ for $-1<\gz<1$ holds.
In conclusion, when $g\leq6$ we have that for all integers $k_1>k_2>0$, every ray in $\gt^+_{sph}$ has an extremal Sasaki metric (up to isotopy).
This improves the result we had in \cite{BoTo13}.

Finally notice that when $g=7$, $k_1=2$, $k_2=1$ and $c=-\frac{299}{301}$, we get that $p(-\frac{1}{5})=-\frac{7794656}{61155675}<0$ and so
positivity of $p(\gz)$ fails in this case.
\end{example}

The case $k_1=k_2$ and $g\leq 6$ was already handled in Example 5.11 of \cite{BoTo13} (recall that $(k_1,k_2)=l(\tilde{w_1},\tilde{w_2})$ in the $S^3_{\tilde{\bfw}}$ join, $M^5_{g,l,\tilde{\bfw}}=S^3_g\star_{l,1}S^3_{\tilde{\bfw}}$). Similarly to the example above we had that every ray in $\gt^+_{sph}$ has an extremal Sasaki metric (up to isotopy). We can thus state the following result.

\begin{proposition}\label{Mkprop}
Let $\bfk=(k_1,k_2)$ with $k_1\geq k_2> 0$ being integers and consider the Yamazaki fiber join $M_\bfk$ as described above.
For $1\leq g \leq 6$ or $1\leq g \leq 1+3k_2$ we have that the entire Sasaki cone is extremal (up to isotopy).
\end{proposition}

\subsection{$N=\bbc\bbp^1\times \bbc\bbp^1$}\label{NNsect}
Let $\grO_i$ denote the standard area forms on the ith copy of $\bbc\bbp^1$. With slight abuse of notation, we denote the pull-back of their K\"ahler classes to $H^2(N,\bbz)$ by $[\grO_1]$ and $[\grO_2]$.
The K\"ahler cone of $N$ then equals $span_{\bbr^+}\{[\grO_1],[\grO_2]\}$. Let $M_\gw$ be a $d=1$ Yamazaki fiber join formed from a choice of K\"ahler classes which are represented by K\"ahler forms 
\begin{equation}\label{kmatrixeqn}
\gro_j=k^1_j\grO_1 +k^2_j\grO_2, \qquad k^1_j,k^2_j\in \bbz^+,
\end{equation}
for $j=1,2$. The line bundles $L_1, L_2$ satisfy that $c_1(L_j)=[\gro_j] = k^1_j[\grO_1] +k^2_j[\grO_2]$.
So the choices of K\"ahler forms is given by the $2$ by $2$ matrix
\begin{equation}\label{Kmatrix}
K=
\begin{pmatrix} 
k^1_1 & k^2_1 \\
 \\
k^1_{2} & k^2_{2},
\end{pmatrix}
\end{equation}
and the fiber join is non-colinear exactly when $det\,K\neq0$.
Now the quotient complex manifold of $M_\gw$  arising from the regular Sasakian structure with Reeb vector field $\xi_{\mathbf 1}$ is equal to 
the following  $\bbc\bbp^1$ bundle over $\bbc\bbp^1\times\bbc\bbp^1$:
\begin{equation}\label{KSprojeqn} \notag
\bbp\bigl(L_1^*\oplus L_2^*) = \bbp\bigl({\BOne} \oplus L_1\otimes L_2^*\bigr)  
= \bbp\bigl({\BOne} \oplus \calo(k^1_1-k^1_2,k^2_1-k^2_2)\bigr) \rightarrow \bbc\bbp^1\times\bbc\bbp^1. 
\end{equation}
We assume here that $k^i_1\neq k^i_2$ for $i=1,2$. If we don't make this assumption our regular quotient could be a product of $\bbc\bbp^1$ with a Hirzebruch surface. This is not a problem per se, but needs to be treated slightly differently, so we will avoid this here.

Every K\"ahler class on $\bbp\bigl({\BOne} \oplus L_1\otimes L_2^*\bigr)$ is admissible in the broader sense of the definition given in \cite{ACGT08}. Thus the fiber join is super admissible and therefore strongly admissible. This case is hence
a special case of Theorem \ref{2ndroundexistence}, with  $\gt^+_{sph}$ a proper subcone of the (unreduced) Sasaki cone. Nevertheless we shall study this example in details since it will illustrate two different approaches for locating
CSC ray(s) in $\gt^+_{sph}$. At the end of the section we will also discuss which polarized K\"ahler manifolds $(S_\bfn,[\omega])$ of the form $S_\bfn=\bbp\bigl({\BOne} \oplus \calo(n_1,n_2)\bigr) \rightarrow \bbc\bbp^1\times\bbc\bbp^1$
appear as regular quotients of a Sasaki Yamazaki fiber join.

For $n_1,n_2 \in \bbz\setminus \{0\}$, a K\"ahler class on the complex manifold $S_\bfn = \bbp\bigl({\BOne} \oplus \calo(n_1,n_2)\bigr) \rightarrow \bbc\bbp^1\times\bbc\bbp^1$ is, up to scale, of the form
$2\pi(\frac{n_1}{x_1}[\grO_1] +\frac{n_2}{x_2}[\grO_2]) + \Xi)$, where $0<|x_i|<1$ and $x_in_i>0$.
As we saw in Section 5.3.3 of \cite{BoTo20b}, as well as in Section \ref{qrquotients} of the current paper, here we can calculate the quotient K\"ahler class up to scale and all in all we get 
a smooth admissible K\"ahler manifold with admissible data 
\begin{equation}\label{regadmdata}
n_1=k^1_1-k^1_2,\quad n_2=k^2_1-k^2_2, \quad x_1=\frac{k^1_1-k^1_2}{k^1_1+k^1_2}, \quad x_2=\frac{k^2_1-k^2_2}{k^2_1+k^2_2}.
\end{equation}
Indeed, more generally, using Proposition \ref{d=1quasiregular} we have that for co-prime $\bfw = (w_1,w_2) \in (\bbz^+)^2$ the quasi-regular quotient of $M_\gw$ with respect to $\xi_\bfw$ is
the log pair \newline
$B_{\gw, \bfw}:=(\bbp({\BOne} \oplus \calo(w_2k^1_1-w_1k^1_2,w_2 k^2_1-w_1 k^2_2)),\Delta_\bfw)$.
Together with the quotient K\"ahler class (up to scale, also from Proposition \ref{d=1quasiregular}) this gives (assuming $w_2k^i_1-w_1k^i_2\neq 0$) admissible data
\begin{equation}\label{qregadmdata}
n_1=w_2k^1_1-w_1k^1_2,\quad n_2=w_2 k^2_1-w_1 k^2_2, \quad x_1=\frac{w_2 k^1_1- w_1 k^1_2}{w_2 k^1_1+w_1k^1_2}, \quad x_2=\frac{w_2k^2_1-w_1k^2_2}{w_2k^2_1+w_1k^2_2}.
\end{equation}
Note that if $w_2k^i_1-w_1k^i_2=0$ for one of (or both) $i=1,2$, we get a product of $\bbc\bbp^1$ with a so-called Hirzebruch orbifold.

From the discussion above we can see  the rays, given up to scale by  $\xi_\bfw$, as $CR$-twists of the regular quotient \cite{ApCa18}. So choosing $c=\frac{w_1-w_2}{w_1+w_2}$ creates a new Sasaki structure via the lift of $f=c\gz+1$ from $S_\bfn$ to $M_\gw$.
With this correspondence in mind we can take two different approaches when seeking out rays in $\gt^+_{sph}$ with constant scalar curvature. From the $CR$-twist point of view, the Reeb vector field $\xi_c$ given by the $CR$-twist has a constant scalar curvature Sasaki metric (up to isotopy) exactly when Equation (10) from \cite{BHLT22} holds. Applying this equation to the regular quotient with admissible data from \eqref{regadmdata} yields the equation $f_{CR}(c)=0$, where

\begin{equation}\label{cscEQinc}
\begin{split}
f_{CR}(c) &  := 18 c \left(c^2-1\right)^2 {k^1_1} {k^2_1} {k^1_2} {k^2_2}+3 (c-1)^5 (k^1_1)^2 (k^2_1)^2+3 (c+1)^5 (k^1_2)^2 (k^2_2)^2\\
&+(c+1) (c-1)^4 k^1_1k^2_1\left(k^1_1 + k^2_1-3   k^2_1{k^1_2}-3  k^1_1  {k^2_2}\right)\\
&+(c+1)^2 (c-1)^3 \left( (k^2_1)^2 {k^1_2} +(k^1_1)^2 {k^2_2}-4  {k^1_1} {k^2_1} {k^1_2}-4 {k^1_1} {k^2_1} {k^2_2}  \right)   \\
&+(c+1)^3 (c-1)^2\left( {k^1_1} (k^2_2)^2  + {k^2_1} (k^1_2)^2-4 {k^2_1} {k^1_2} {k^2_2}-4  {k^1_1} {k^1_2} {k^2_2} \right) \\
&+(c+1)^4 (c-1)k^1_2k^2_2\left( k^1_2 +  k^2_2-3  {k^1_1} k^2_2 -3  {k^2_1} k^1_2  \right)\\
  \end{split}
\end{equation}

If $c\in \bbq\cap (-1,1)$, we can then set $c=\frac{w_1-w_2}{w_1+w_2} $ to get an equation in $(w_1,w_2)\in \bbz^+\times \bbz^+$ for $\xi_\bfw$ being CSC (up to isotopy):
\begin{equation}\label{cscfinal}
\begin{split}
0&= -3 (k^1_2)^2 (k^2_2)^2 w_1^5 \\
&+ k^1_2 k^2_2 \left(k^1_2 + k^2_2 - 3 k^2_1 k^1_2 - 3 k^1_1 k^2_2\right)w_1^4 w_2 \\
&+\left(4 k^1_1 k^1_2 k^2_2 + 4 k^2_1 k^1_2 k^2_2   - 9 k^1_1 k^2_1 k^1_2 k^2_2  -k^2_1 (k^1_2)^2 -k^1_1 (k^2_2)^2\right) w_1^3 w_2^2\\
&+\left(9 k^1_1 k^2_1 k^1_2 k^2_2   + (k^2_1)^2 k^1_2 
 + (k^1_1)^2 k^2_2  - 4 k^1_1 k^2_1 k^2_2- 4 k^1_1 k^2_1 k^1_2
 \right)w_1^2 w_2^3\\
&+ k^1_1 k^2_1 \left(3 k^2_1 k^1_2 + 3 k^1_1 k^2_2-k^1_1 - k^2_1\right)w_1 w_2^4 \\
&+ 3(k^1_1)^2 (k^2_1)^2 w_2^5.
\end{split}
  \end{equation}

On the other hand, Proposition 4.13 of \cite{BoTo21} 
(with $m_0=w_1$, $m_\infty=w_2$, $r_1=x_1$, and $r_2=x_2$) tells us that the K\"ahler class given by $(x_1,x_2)$ on the log pair $(S_\bfn,\Delta_\bfw)$ has a
constant scalar curvature K\"ahler metric when the following equation holds true:
\begin{equation}\label{cscEQinw}
\begin{split}
0= & 9 (w_1 - w_2) n_1 n_2 - 6 (w_1 + w_2) n_1 n_2 (x_1 + x_2) + 6 (w_1 - w_2) n_1 n_2 x_1x_2\\
&+ 3n_2 (4 w_1 w_2 - n_1 (w_1 - w_2)) x_1^2 +
 3n_1 (4 w_1 w_2 - n_2 (w_1 - w_2))  x_2^2 \\
 & - (4 w_1 w_2 (n_1 + n_2) - 
    3 (w_1 - w_2) n_1 n_2) x_1^2 x_2^2.
    \end{split}
\end{equation}
We can then use the data in \eqref{qregadmdata} above to get an equation for the existence of a constant scalar curvature K\"ahler metric in the K\"ahler class of the 
quasi-regular K\"ahler quotient of $\xi_\bfw$. 

As expected from the above discussion and the fact that a quasi-regular Sasaki structure has constant scalar curvature (up to isotopy) exactly when its 
K\"ahler quotient has a constant scalar curvature K\"ahler metric in its K\"ahler class, this gives an equation equivalent to \eqref{cscfinal}. 

Consider a given complex manifold $S_\bfn = \bbp\bigl({\BOne} \oplus \calo(n_1,n_2)\bigr) \rightarrow \bbc\bbp^1\times\bbc\bbp^1$. This will be the regular quotient of a $d=1$ Yamazaki fiber join given by
$K$ for any matrix $K$ of the form
\begin{equation}
K=
\begin{pmatrix} 
n_1+k^1 & n_2+k^2 \\
 \\
k^1 & k^2,
\end{pmatrix}
\end{equation}
where $k^i \in \bbz$ such that $k^i>Max\{0,-n_i\}$.
For a given choice of $k^1,k^2$, the quotient K\"ahler class is then determined, up to scale, by $x_1=\frac{n_1}{n_1+2k^1}$ and $x_2=\frac{n_2}{n_2+2k^2}$.
This gives a criterion for which K\"ahler classes on $S_\bfn$ can show up as regular quotient K\"ahler classes of a $d=1$ Yamazaki fiber join.

For example, if $n_1=1$ and $n_2=-1$, we have $x_1=\frac{1}{1+2k^1}$ and $x_2=\frac{-1}{-1+2k^2}$. Here $k^1 \in \bbz^+$ and $k^2 \in \bbz^+\setminus\{1\}$.
The Koiso-Sakane KE class is given by $x_1=1/2$ and $x_2=-1/2$ and we see right away that this class is out of range. The other CSC classes on this manifold are
given by $x_2=-x_1$ and $x_2=x_1-1$ (see e.g. Theorem 9 in \cite{ACGT08}).
Now,
$$
\begin{array}{ccl}
x_2 & = &-x_1\\
\\
&\iff&
\\
\frac{-1}{-1+2k^2} & = & -\frac{1}{1+2k^1}\\
\\
&\iff&
\\
k^2&=& k^1+1,
\end{array}
$$
which then gives us a one parameter family $(x_1,x_2) = (\frac{1}{1+2k^1},\frac{-1}{1+2k^1})$ , $k^1\in \bbz^+$ of
CSC K\"ahler classes that each are regular quotient K\"ahler classes of a $d=1$ Yamazaki fiber join.

One the other hand,
$$
\begin{array}{ccl}
x_2 & = &x_1-1\\
\\
&\iff&
\\
\frac{-1}{-1+2k^2} & = & \frac{-2k^1}{1+2k^1}\\
\\
&\iff&
\\
1&=& -4k^1-4k^1 k^2,
\end{array}
$$
which has no solutions for $k^1 \in \bbz^+$ and $k^2 \in \bbz^+\setminus\{1\}$. Thus, none of the CSC K\"ahler classes from this family can be regular quotient K\"ahler classes of a $d=1$ Yamazaki fiber join.

\subsection{$N= \Sigma_{g_1}\times\Sigma_{g_2}$, a product of Riemann surfaces}\label{highergenusprod}
We can generalize the example of Section \ref{NNsect} to consider the case where $N=\Sigma_{g_1}\times\Sigma_{g_2}$ with $\Sigma_{g_i}$ each being compact Riemann surfaces of genus $g_i$,
equipped with a standard CSC area form $\grO_i$. Similarly to Section \ref{NNsect}, each choice of matrix 
$K=
\begin{pmatrix} 
k^1_1 & k^2_1 \\
 \\
k^1_{2} & k^2_{2}.
\end{pmatrix}
$,
consisting of positive integer entries $k^i_j$, yields a $d=1$ Yamazaki fiber join $M_\gw=S(L_1^*\oplus L_2^*)$ via the line bundles 
$L_1, L_2$ satisfying $c_1(L_j)=[\gro_j] = k^1_j[\grO_1] +k^2_j[\grO_2]$. We assume here that $k^i_1\neq k^i_2$ for $i=1,2$.

The case that $M$ is the total space of a Sasakian fiber join with $N=\grS_{g_1}\times\grS_{g_2}$ was treated in Proposition 5.8 of \cite{BoTo20b}. When $d>1$ the spectral sequence of the fibration collapses, so the cohomology groups of $M$ are the cohomology groups of the product $S^{2d+1}\times \grS_{g_1}\times\grS_{g_2}$. When $d=1$ we have
\begin{equation}\label{HM}
H^p(M^7,\bbz)= \begin{cases} \bbz~&\text{if $p=0,7$} \\
                               \bbz^{2g_1+2g_2} ~&\text{if $p=1,3,6$} \\
                              \bbz^{4g_1g_2+2}~ &\text{if $p=2,5$} \\
                              \bbz^{2g_1+2g_2} +\bbz_e~ &\text{if $p=4$} \\
                              0 &\text{otherwise}
                               \end{cases} 
\end{equation}
where the image of the differential $d_4$ in $E^{4,0}_2$ is the Euler class of the bundle with $e=k^1_1k^2_2+k^2_1k^1_2$. In both cases with $g_1,g_2$ and $e$ fixed we know that $H^4(N,\bbz)=\bbz$, so it follows from a theorem of Pontrjagin \cite{Pon45} (see also \cite{Mas58,DoWh59}) that the sphere bundles $M$ are classified by their 2nd and 4th Stiefel-Whitney classes $w_2,w_4$, and their Pontrjagin class $p_1(M)$.

Similarly to Section \ref{NNsect}, the quotient complex manifold of $M_\gw$  arising from the regular Sasakian structure with Reeb vector field $\xi_{\mathbf 1}$ is equal to 
the following  $\bbc\bbp^1$ bundle over $\Sigma_{g_1}\times\Sigma_{g_2}$:
$$
\bbp\bigl(L_1^*\oplus L_2^*) = \bbp\bigl({\BOne} \oplus L_1\otimes L_2^*\bigr)  
= \bbp\bigl({\BOne} \oplus \calo(n_1,n_2)\bigr) \rightarrow \Sigma_{g_1}\times\Sigma_{g_2}, 
$$
with $n_1=k^1_1-k^1_2$ and $n_2=k^2_1-k^2_2$.
Further, the regular quotient K\"ahler class is, up to scale, equal to the admissible
K\"ahler class  $2\pi(\frac{n_1}{x_1}[\grO_1] +\frac{n_2}{x_2}[\grO_2]) + \Xi)$ where
$x_1=\frac{k^1_1-k^1_2}{k^1_1+k^1_2}, \quad x_2=\frac{k^2_1-k^2_2}{k^2_1+k^2_2}$.

When $g_i\geq 2$ for at least one of $i=1,2$, we cannot use Theorem \ref{2ndroundexistence} to get existence of extremal/CSC Sasaki metrics. Further, we know from the examples in Sections 3.3 and 3.4 of \cite{BHLT22} that the existence of CSC or even just extremal Sasaki metrics is by no means a given. More specifically, Proposition 3.10 of \cite{BHLT22} tells us that the Reeb vector field determined - up to homothety - by $c\in (-1,1)$ (as explained in the beginning of the section) is extremal (up to isotopy)
if and only if $F_c(\gz)>0$, for $-1<\gz <1$, where the polynomial
$F_c(\gz)$ is given as follows:

\noindent
Let $s_i=\frac{2(1-g_i)}{n_i}=\frac{2(1-g_i)}{k^i_1-k^i_2}$,  $x_i=\frac{k^i_1-k^i_2}{k^i_1+k^i_2}$, and define
$$
\begin{array}{ccl}
\alpha_{r,-5} & = & \int_{-1}^1(ct+1)^{-5}t^r(1+x_1t)(1+x_2t)\,dt\\
\\
\alpha_{r,-6} & = & \int_{-1}^1(ct+1)^{-6}t^r(1+x_1t)(1+x_2t)\,dt\\
\\
\beta_{r,-4} & = & \int_{-1}^1(ct+1)^{-4}t^r(x_1s_1(1+x_2 t)+x_2 s_2(1+x_1 t))\,dt \\
\\
&+ & (-1)^r(1-c)^{-4}(1-x_1)(1-x_2) +(1+c)^{-4}(1+x_1)(1+x_2).
\end{array}
$$
Then, 
\begin{equation}\label{wextrpol7mnf}
F_c(\gz)=(c\gz+1)^4 \left[ \frac{2(1-x_1)(1-x_2)}{(1-c)^4}(\gz+1) + \int_{-1}^\gz Q(t)(\gz-t)\,dt\right],
\end{equation}
where 
$$Q(t) = \frac{2\left( x_1s_1(1+x_2t)+x_2s_2(1+x_1t)\right)}{(ct+1)^4} - \frac{(A_1 t+A_2)(1+x_1t)(1+x_2t)}{(ct+1)^6}$$
and $A_1$ and $A_2$ are the unique solutions to the linear system
\begin{equation}\label{wextrpol7mnf2}
\begin{array}{ccl}
\alpha_{1,-6}A_1+\alpha_{0,-6}A_2&=& 2\beta_{0,-4}\\
\\
\alpha_{2,-6}A_1+\alpha_{1,-6}A_2&=& 2\beta_{1,-4}.
\end{array}
\end{equation}
Further, if the positivity of $F_c(\gz)$ is satisfied, then the extremal Sasaki structure is CSC exactly when
\begin{equation}\label{scsc7mnf}
\alpha_{1,-5}\beta_{0,-4} - \alpha_{0,-5}\beta_{1,-4}=0
\end{equation}
is satisfied. A direct calculation shows that
$\alpha_{1,-5}\beta_{0,-4} - \alpha_{0,-5}\beta_{1,-4}=\frac{4h(c)}{9(1-c^2)^7}$ where $h(c)$ is the polynomial given by
\small
$$
\begin{aligned}
h(c)&=(3 x_1 x_2( s_1 x_2 + s_2x_1)- s_1 x_1- s_2 x_2+3(3  x_1^2 x_2^2-  x_1^2+2  x_1 x_2-  x_2^2+1)) c^5\\
&+(s_1 x_1^2 +s_2 x_2^2- 3  (s_1+s_2) x_1^2 x_2^2- 4  (s_1+s_2) x_1 x_2-6 (x_1+x_2)(4x_1 x_2+1))c^4 \\
&+4(((s_1x_1+s_2x_2)-  (s_1x_2+s_2x_1)) x_1 x_2 +  s_1 x_1+ s_2 x_2+3 x_1 x_2(x_1 x_2+5)  +6 (x_1^2+x_2^2))c^3\\
&+4((s_1+s_2) (x_1x_2+ 1) x_1 x_2-  s_1 x_1^2- s_2 x_2^2-3(x_1+x_2)(2x_1x_2+3 ))c^2\\
&+((s_1 x_2+s_2 x_1)x_1 x_2-  (s_1 x_1+s_2 x_2)(4x_1 x_2+3) +3 (x_1^2 x_2^2+  x_1^2+ x_2^2+10  x_1 x_2+7)) c\\
&+3 (s_1 x_1^2+ s_2 x_2^2) -(s_1+s_2) x_1^2 x_2^2-6 (x_1+ x_2)
\end{aligned}
$$

\normalsize
and $h(\pm 1)=\pm 24(1\mp x_1)^2(1\mp x_2)^2$. Thus, equation \eqref{scsc7mnf} always have a solution $c\in (-1,1)$.
In the event that $g_1,g_2\leq 1$, this is predicted by (the proof of) Theorem \ref{bhlt22}  and in the event that $g_1,g_2\geq 1$ (where $\gt^+$ is $2$-dimensional) this is predicted by Corollary 1.7 of \cite{BHL17}.
The $g_1=0$ and $g_2>1$ (or vice versa) case falls outside of these results. Of course, a solution to $h(c)=0$ only corresponds to an actual CSC ray if we also have that the positivity condition of $F_c(\gz)$ is satisfied.

\begin{proposition}\label{2highergenus}
Let $M_\gw$ be a $d=1$ fiber join over $\grS_{g_1}\times \grS_{g_2}$ with its induced Sasakian structure. Then for all $g_1, g_2 \geq 1$, there exists a matrix $K=\tiny
\begin{pmatrix} 
k^1_1 & k^2_1 \\
 \\
k^1_{2} & k^2_{2}.
\end{pmatrix}$ \normalsize such that the entire Sasaki cone of $M_\gw$ is extremal and contains a CSC ray.
\end{proposition}

\begin{proof}
Without loss of generality, we assume that $g_2\geq g_1\geq 1$.
First we note that since $g_1,g_2\geq 1 $, the Sasaki cone is of dimension $2$. Thus, the proof will consist of showing that for all $g_2\geq g_1\geq 1$, $\exists$ a two-by-two matrix $K$ such that $\forall c\in (-1,1)$, $F_c(\gz)$ as defined in
\eqref{wextrpol7mnf} is positive for $-1< \gz <1$. Once this is proven we already know from the above discussion that for this such a choice of $K$, \eqref{scsc7mnf} has a solution $c\in (-1,1)$. This $c$ will correspond to a CSC ray.
If $g_1=g_2=1$, the result follows from Theorem \ref{2ndroundexistence}.
Thus, we assume for the rest of the proof that $g_2>1$.
Now, let $K=\tiny
\begin{pmatrix} 
10g_1& 100g_2 \\
 \\
2 g_1 & g_2.
\end{pmatrix}$. 
Using \eqref{wextrpol7mnf}, we can calculate that
$$F_c(\gz)=\frac{(1-\gz^2)p(\gz)}{1212g_1g_2 h_0(c)},$$
where
$$h_0(c)=544829 - 1814364 c + 2225984 c^2 - 1185624 c^3 + 229199 c^4,$$
and $p(\gz)$ is a cubic given by
$$p(\gz)=8g_1g_2h_1(c)+4h_2(c,g_1,g_2)(1+\gz)+2h_3(c,g_1,g_2)(1+\gz)^2+h_4(c,g_1,g_2)(1+\gz)^3,$$
where
$$
\begin{array}{ccl}
h_1(c)&=&h_0(c)\\
\\
h_2(c,g_1,g_2) &=&  6h_{21}(c)+h_{22}(c)(g_2-2)+\left(5h_{23}(c)+h_{24}(c)(g_2-2)\right)(g_1-2)\\
\\
h_3(c,g_1,g_2) &=&  2h_{31}(c)+2h_{32}(c)(g_2-2)+\left(h_{33}(c)+2h_{34}(c)(g_2-2)\right)(g_1-2)\\
\\
h_4(c,g_1,g_2) &=&  10h_{41}(c)+h_{42}(c)(g_2-2)+\left(2h_{43}(c)+h_{44}(c)(g_2-2)\right)(g_1-2)\end{array}
$$
with
$$
\begin{array}{ccl}
h_{21}(c) &=&1849633 - 3952908 c + 2583653 c^2 - 545438 c^3 + 68368 c^4\\
\\
h_{22}(c) &=&5029446 - 10073556 c + 5505031 c^2 - 421486 c^3 - 29519 c^4\\
\\
h_{23}(c) &=&1085299 - 2250304 c + 1327594 c^2 - 148704 c^3 - 11901 c^4\\
\\
h_{24}(c) &=&2453521 - 4733176 c + 2196021 c^2 + 235654 c^3 - 147064 c^4\\
\\
h_{31}(c) &=&173925883 - 629489348 c + 863749558 c^2 - 530449308 c^3 + 
 122385903 c^4\\
 \\
h_{32}(c) &=&86771822 - 314540932 c + 432305747 c^2 - 265928422 c^3 + 61453077 c^4\\
\\
h_{33}(c) &=&169929491 - 609982556 c + 828678836 c^2 - 502956696 c^3 + 
 114452421 c^4\\
 \\
h_{34}(c) &=&42386813 - 152393768 c + 207385193 c^2 - 126091058 c^3 + 28743168 c^4\\
\\
h_{41}(c) &=&72852912 - 233877440 c + 270006303 c^2 - 130233426 c^3 + 21229919 c^4\\
\\
h_{42}(c) &=&365166252 - 1171579852 c + 1351415507 c^2 - 650974422 c^3 + 105863967c^4\\
\\
h_{43}(c) &=&184191678 - 594750598 c + 693107613 c^2 - 339776268 c^3 + 57173843 c^4\\
\\
h_{44}(c) &=&184642524 - 595846924 c + 693799609 c^2 - 339679914 c^3 + 57031029 c^4.
\end{array}
$$
We also notice that $p(1)=4000g_1g_2 h_0(c)$.

\noindent
{\em Claim:} For all $c\in (-1,1)$, $h_0(c)>0$. Further, for all $c\in (-1,1)$, $i=2,3$, and $j=1,2,3,4$, $h_{ij}(c)>0$. 

From this claim it then follows that for $g_1,g_2>1$, all $c\in (-1,1)$, and $i=0,1,2,3$, $h_i(c)>0$.
Thus, in this case, we have $p(\pm 1)>0$, $p'(-1)>0$, and $p''(-1)>0$. Since $p(\gz)$ is a cubic, a moment's thought tells us that
$p(\gz)>0$ for $-1<\gz <1$. Finally, since the claim also tells us that $h_0(c)>0$ for $c\in (-1,1)$, we conclude that
$F_c(\gz)$ is positive for all $c\in (-1,1)$ and $\gz \in (-1,1)$ as desired.

The proof of the claim 
is a standard exercise: For example, one easily checks that
$h_0(\pm 1)>0$, $h_0'(\pm 1)<0$. Further, since  $h_0''(c)$ is a second order polynomial in $c$ with $h_0''(\pm 1)>0$, and $h_0'''(1)<0$,
we know $h_0''(c)>0$ for $c\in (-1,1)$.
Thus for $-1\leq c \leq 1$, $h_0(c)$, is a (concave up and) decreasing function that is positive at $c=\pm 1$. It therefore must be positive for all
$c\in (-1,1)$, as desired. The argument for the claim concerning $h_{ij}(c)$ with $i=2,3$ and $j=1,2,3,4$ is completely similar.

Finally, if $g_1=1$ (and $g_2>1$), 
we still have that $h_0(c)=h_1(c)>0$ for $c\in (-1,1)$. Further,
note that
$$
\begin{array}{ccl}
h_2(c,1,g_2) &=&  \tilde{h}_{21}(c)+5\tilde{h}_{22}(c)(g_2-2)\\
\\
h_3(c,1,g_2) &=&  5\tilde{h}_{31}(c)+2\tilde{h}_{32}(c)(g_2-2),
\end{array}
$$
where
$$
\begin{array}{ccl}
\tilde{h}_{21}(c) &=&5671303 - 12465928 c + 8863948 c^2 - 2529108 c^3 + 469713 c^4\\
\\
\tilde{h}_{22}(c) &=&515185 - 1068076 c + 661802 c^2 - 131428 c^3 + 23509 c^4\\
\\
\tilde{h}_{31}(c) &=&35584455 - 129799228 c + 179764056 c^2 - 111588384 c^3 + 26063877 c^4\\
 \\
\tilde{h}_{32}(c) &=&44385009 - 162147164 c + 224920554 c^2 - 139837364 c^3 + 32709909 c^4
\end{array}
$$
Now, in an exact similar way as above, we can prove that for all $c\in (-1,1)$, $i=2,3$, and $j=1,2$, $\tilde{h}_{ij}(c)>0$.
Therefore we may still conclude that $p(\pm 1)>0$, $p'(-1)>0$, and $p''(-1)>0$ and the proof finishes as above.

\end{proof}

\subsection{$N=\bbp(E) \rightarrow \Sigma_g$, where $E\rightarrow \Sigma_g$ is a polystable rank $2$ holomorphic vector bundle over a compact Riemann surface of genus $g\geq 1$.}

Let $\Sigma_g$ be a compact Riemann surface and let $E\rightarrow \Sigma_g$ be a holomorphic vector bundle. The degree of $E$, is defined by
$deg\,E=\int_{\Sigma_g} c_1(E)$. Then $E$ is {\em stable}(or {\em semistable}) in the sense of Mumford if for any proper coherent
subsheaf $F$, $\frac{deg\, F}{rank\,F} < \frac{deg\, E}{rank\,E}$ (or $\frac{deg\, F}{rank\,F} \leq \frac{deg\, E}{rank\,E}$).
Further, a semistable holomorphic vector bundle, $E$, is called {\em polystable} if it decomposes as a direct sum of stable holomorphic vector bundles, $E=F_1\oplus \cdots\oplus F_l$, such that that $\frac{deg\, F_i}{rank\,F_i} = \frac{deg\, E}{rank\,E}$, for $i=1,\dots,l$.
(See e.g. \cite{Kobbook} for more details on this.)

Assume  $N=\bbp(E) \stackrel{\pi}{\rightarrow} \Sigma_g$, where $E\rightarrow \Sigma_g$ is a polystable rank $2$ holomorphic vector bundle over a compact Riemann surface of genus $g\geq 1$. 
Note that the polystabilty of $E$ is independent of the choice of $E$ in $\bbp(E)$. 
Indeed, by the theorem of Narasimhan and Seshadri \cite{NaSe65}, polystability of $E$ is equivalent to $\bbp(E) \stackrel{\pi}{\rightarrow} \Sigma_g$ admitting a flat projective unitary connection which in turn is equivalent to $N$ admitting a local product K\"ahler metric induced by constant scalar curvature K\"ahler metrics on $\Sigma_g$ and $\bbc\bbp^1$. We shall explain and explore the latter in more detail below. Likewise,
the condition of whether $deg E$ is even ($E$ spin) or odd ($E$ is non-spin), is independent of the choice of $E$. Unless $E$ is decomposable, we must have that $Aut(N,J)$ is discrete (\cite{Mar71}).

Let ${\mathbf v}=c_1(VP(E))\in H^2(N,\bbz)$ denoted the Chern class of the vertical line bundle and let ${\mathbf f}\in H^2(N,\bbz)$ denote the Poincar\'e dual of the fundamental class of a fiber of $\bbp(E) \rightarrow \Sigma_g$. From e.g. \cite{Fuj92} we know that if ${\mathbf h}\in  H^2(N,\bbz)$ denote the Chern class of the ($E$-dependent) tautological line bundle on $N$, then
$H^2(N,\bbz)=\bbz {\mathbf h} \oplus \bbz {\mathbf f}$ and ${\mathbf v} = 2{\mathbf h}+(deg E){\mathbf f}$.

Due to the fact that $N=\bbp(E) \stackrel{\pi}{\rightarrow} \Sigma_g$ admits a flat projective unitary connection, we know that $N$ has a universal cover $\tilde{N} =\bbc\bbp^1\times\tilde{\Sigma_g}$ (where $\tilde{\Sigma_g}$ is the universal cover of $\Sigma_g$). 
Let $\Omega_1$ denote the standard Fubini-Study area form on $\bbc\bbp^1$ and let $\Omega_2$ denote a standard CSC area form on $\Sigma_g$. Now consider the projection $\pi_1: \bbc\bbp^1\times\tilde{\Sigma_g}\rightarrow \bbc\bbp^1$ to the first factor. Then $\pi_1^*(\Omega_1)$ descends to a closed $(1,1)$ form on $N$ representing the class ${\mathbf v}/2$ and $[\pi^*\Omega_2]={\mathbf f}$.

If we (abuse notation slightly and) think of $q_1\Omega_1+q_2\Omega_2$ as a local product of CSC K\"ahler forms on $N$, then this represents the cohomology class 
$\frac{q_1}{2}{\mathbf v} + q_2{\mathbf f}=q_1{\mathbf h} +(\frac{q_1}{2}(deg E)+q_2){\mathbf f}$. If $deg E$ is even, this class is in $H^2(N,\bbz)$ (and hence can represent a holomorphic line bundle) precisely when $q_1,q_2\in \bbz$. If $deg E$ is odd, then the class is in $H^2(N,\bbz)$  iff ($q_1$ is an even integer and $q_2\in \bbz$) or
($q_1$ is an odd integer and $(q_2-1/2)\in \bbz$). Note that a similar discussion appears in the proof of Theorem 4.6 of \cite{ACGT08b}.

With this in mind, we can we can (yet again) generalize to consider the case where $N$ is as described above.
We consider a matrix 
$K=
\begin{pmatrix} 
k^1_1 & k^2_1 \\
 \\
k^1_{2} & k^2_{2}.
\end{pmatrix}
$,
consisting of entries $k^i_j$, such that:
\begin{itemize}
\item If $deg E$ is even, $k^i_j \in \bbz^+$
\item If $deg E$ is odd, one of the following is true:
\begin{itemize}
\item  $k^1_j$ is an even positive integer and $k^2_j\in \bbz^+$ 
\item $k^1_j$ is an odd positive integer and $(k^2_j-1/2)\in \bbz^+$.
\end{itemize}
\end{itemize}
Such a choice of $K$ yields
a $d=1$ Yamazaki fiber join $M_\gw=S(L_1^*\oplus L_2^*)$ via the line bundles 
$L_1, L_2$ satisfying $c_1(L_j)=[\gro_j] = k^1_j[\grO_1] +k^2_j[\grO_2]=k^1_j{\mathbf h} +(\frac{k^1_j}{2}(deg E)+k^2_j){\mathbf f}$. As before we assume that $k^i_1\neq k^i_2$ for $i=1,2$.

As we know, the quotient complex manifold of $M_\gw$  arising from the regular Sasakian structure with Reeb vector field $\xi_{\mathbf 1}$ is equal to 
the following  $\bbc\bbp^1$ bundle over $N$:
$
\bbp\bigl(L_1^*\oplus L_2^*) = \bbp\bigl({\BOne} \oplus L_1\otimes L_2^*\bigr), 
$
with 
$c_1( L_1\otimes L_2^*)=(k^1_1-k^1_2)[\grO_1] +(k^2_1-k^2_2)[\grO_2]=(k^1_1-k^1_2){\mathbf h} +(\frac{(k^1_1-k^1_2)}{2}(deg E)+(k^2_1-k^2_2)){\mathbf f}$.

Similarly, as before, the regular quotient K\"ahler class is, up to scale, equal to the admissible
K\"ahler class  $2\pi(\frac{k^1_1-k^1_2}{x_1}[\grO_1] +\frac{k^2_1-k^2_2}{x_2}[\grO_2]) + \Xi)$ where
$x_1=\frac{k^1_1-k^1_2}{k^1_1+k^1_2}, \quad x_2=\frac{k^2_1-k^2_2}{k^2_1+k^2_2}$.

We can now adapt the set-up from Section \ref{highergenusprod} with $s_1=\frac{2}{k^1_1-k^1_2}$ and $s_2=\frac{2(1-g)}{k^2_1-k^2_2}$.
In particular, equation \eqref{scsc7mnf} continues to have some solution $c\in (-1,1)$ and we can calculate $F_c(\gz)$ using \eqref{wextrpol7mnf}. 
If a choice of $K$ satisfies that $F_c(\gz)$ is positive for all $c\in (-1,1)$ and $\gz \in (-1,1)$, then we will have a  conclusion similar to the result in 
Proposition \ref{2highergenus}.  Indeed, we have the following proposition. 

\begin{proposition}\label{polyprop}
Let $N=\bbp(E) \stackrel{\pi}{\rightarrow} \Sigma_g$, where $E\rightarrow \Sigma_g$ is a polystable rank $2$ holomorphic vector bundle over a compact Riemann surface of genus $g\geq 1$. Let 
$K=\tiny
\begin{pmatrix} 
k^1_1 & k^2_1 \\
 \\
k^1_{2} & k^2_{2}
\end{pmatrix} = \begin{pmatrix} 
10 g & 100 g\\
 \\
2g & g
\end{pmatrix}$ \normalsize
and let $M_\gw$ be the $d=1$ fiber join over $N$ as described above with its induced Sasakian structure. Then the entire subcone, $\gt^+_{sph}$, is extremal and contains a CSC ray.

In particular, if $E$ is indecomposable, then the entire Sasaki cone of $M_\gw$ is extremal and contains a CSC ray.
\end{proposition}

\begin{proof}
First we notice that $k^1_j$ is even for $j=1,2$ and thus this choice of $K$ is allowed whether or not $E$ is spin. Second, we have that the set of rays
in $\gt^+_{sph}$ is parametrized by $c\in (-1,1)$ in the same manner as in Section \ref{highergenusprod}.
Further, in the case where $E$ is indecomposable, $Aut(N,J)$ is discrete and thus the Sasaki cone is exactly $\gt^+_{sph}$.
Therefore all we need to do to prove the proposition is to check that for this choice of $K$, the polynomial
$F_c(\gz)$, defined by \eqref{wextrpol7mnf}, is positive for all $c\in (-1,1)$ and $\gz \in (-1,1)$. 

If $g=1$, we already know from (the proof of) Theorem 3.1 in \cite{ApMaTF18} that for any choice of $K$, $F_c(\gz)>0$ for all $c\in (-1,1)$ and $\gz \in (-1,1)$. 
Thus we will assume that $g>1$ for the rest of the proof.

By direct calculations we get that
$$F_c(\gz) = \frac{(1-\gz^2)p(\gz)}{1212 g h_0(c)},$$
where
$$h_0(c) = 544829 - 1814364 c + 2225984 c^2 - 1185624 c^3 + 229199 c^4$$
and $p(\gz)$ is a cubic in $\gz$ that we may write as 
$$
\begin{array}{ccl}
p(\gz) & = & 8 g h_1(c) + \left(4h_{21}(c)+20h_{22}(c)(g-2)\right)(\gz+1) + \left(2h_{31}(c)+4h_{32}(c)(g-2)\right)(\gz+1)^2 \\
\\
&+& \left(h_{41}(c)+2h_{42}(c)(g-2)\right)(\gz+1)^3,
\end{array}
$$
where
$$
\begin{array}{ccl}
h_1(c)&=&h_0(c)\\
\\
h_{21}(c)&=&5793707 - 13073132 c + 9976937 c^2 - 3421902 c^3 + 734322 c^4\\
\\
h_{22}(c)&=&515185 - 1068076 c + 661802 c^2 - 131428 c^3 + 23509 c^4\\
\\
h_{31}(c)&=&181918667 - 668502932 c + 933891002 c^2 - 585434532 c^3 + 
 138252867 c^4\\
\\
h_{32}(c)&=&44385009 - 162147164 c + 224920554 c^2 - 139837364 c^3 + 32709909 c^4\\
\\
h_{41}(c)&=&356026968 - 1129159208 c + 1277664093 c^2 - 594396318 c^3 + 
 89753413 c^4\\
\\
h_{42}(c)&=&90261864 - 287866464 c + 328807949 c^2 - 155647254 c^3 + 24416469 c^4.\\
\end{array}
$$
Note also that $p(1)=4000gh_0(c)$.

Completely similar to the way the claim at the end of the proof of Proposition \ref{2highergenus} is verified, we can now show that
for all $c\in (-1,1)$, $h_0(c)>0$ and for all $c\in (-1,1)$, $i=2,3$, and $j=1,2$, $h_{ij}(c)>0$. 
This tells us that $p(\pm 1)>0$, $p'(-1)>0$, and $p''(-1)>0$. Since $p(\gz)$ is a cubic, we conclude that
$p(\gz)>0$ for $-1<\gz <1$. Finally, since $h_0(c)>0$ for $c\in (-1,1)$, 
$F_c(\gz)$ is positive for all $c\in (-1,1)$ and $\gz \in (-1,1)$ as desired.

\end{proof}

\begin{remark}
Note that if we fix a matrix $K$ and calculate $F_c(\gz)$, then we can observe that
$$\lim_{g\rightarrow +\infty} F_0(0)=-\infty.$$
Thus is it clear that for any choice of $K$ there exist values $g>1$ such that 
the corresponding Sasaki cone is NOT exhausted by extremal Sasaki metrics.

Experimenting with Mathematica, it seems that e.g. choosing $K=\begin{pmatrix} 
4 g & 3 g\\
 \\
2g & g
\end{pmatrix}$ would also yield a $F_c(\gz)$ such that $F_c(\gz)$ is positive for all $c\in (-1,1)$ and $\gz \in (-1,1)$,
but the argument would be relying on using Mathematica to calculate the numerical values of the real roots of certain fourth degree polynomials.
For the sake of a transparent argument we chose a more optimal $K$ to do the job in the proof above.
\end{remark}

\def\cprime{$'$} \def\cprime{$'$} \def\cprime{$'$} \def\cprime{$'$}
  \def\cprime{$'$} \def\cprime{$'$} \def\cprime{$'$} \def\cprime{$'$}
  \def\cdprime{$''$} \def\cprime{$'$} \def\cprime{$'$} \def\cprime{$'$}
  \def\cprime{$'$}
\providecommand{\bysame}{\leavevmode\hbox to3em{\hrulefill}\thinspace}
\providecommand{\MR}{\relax\ifhmode\unskip\space\fi MR }
\providecommand{\MRhref}[2]{%
  \href{http://www.ams.org/mathscinet-getitem?mr=#1}{#2}
}
\providecommand{\href}[2]{#2}

\end{document}